\documentclass[12pt,reqno,a4paper]{amsart}
\usepackage[utf8]{inputenc}

\usepackage{graphicx, amsmath, amssymb, amscd, amsthm, euscript, psfrag, amsfonts,bm}

\usepackage{enumitem}
\usepackage{etoolbox}
\usepackage{amsmath}
\usepackage{enumerate}
\usepackage{amssymb}
\usepackage{amscd}
\usepackage{amsthm}
\usepackage{amsfonts}
\usepackage{graphicx}
\usepackage{tensor}
\usepackage{multirow}
\usepackage{multicol}
\usepackage{mathscinet}
\usepackage[T1]{fontenc}
\usepackage{mathrsfs}

\usepackage[hidelinks]{hyperref}
\usepackage{cleveref}

\makeatletter
\newtheorem*{rep@theorem}{\rep@title}
\newcommand{\newreptheorem}[2]{%
\newenvironment{rep#1}[1]{%
 \def\rep@title{#2 \ref{##1}}%
 \begin{rep@theorem}}%
 {\end{rep@theorem}}}
\makeatother

\crefformat{section}{\S#2#1#3}
\crefformat{subsection}{\S#2#1#3}
\crefformat{subsubsection}{\S#2#1#3}

\newtheorem{theorem}{Theorem}[section]
\newreptheorem{theorem}{Theorem}
\newtheorem{lemma}[theorem]{Lemma}
\newtheorem{proposition}[theorem]{Proposition}
\newtheorem{corollary}[theorem]{Corollary}

\newtheorem*{genericthm*}{\thistheoremname}
\newenvironment{namedthm*}[1]
  {\renewcommand{\thistheoremname}{#1}
   \begin{genericthm*}}
  {\end{genericthm*}}
\theoremstyle{definition}
\newtheorem{definition}[theorem]{Definition}

\theoremstyle{remark}
\newtheorem{remark}[theorem]{Remark}

\newtheorem{question}[theorem]{Question}

\newcommand{\bR}{\mathbb{R}}

\newcommand{\prm}{^{\prime}}
\numberwithin{equation}{section}

\newcommand{\bZ}{\mathbb{Z}}

\newcommand{\norm}[1]{\left\lVert#1\right\rVert}

\begin{document}
\title[Genericity on submanifolds]{Genericity on submanifolds and application to Universal hitting time statistics}
\author{Han Zhang }
\address{Yau Mathematical Sciences Center, Tsinghua University, Beijing, 100084, China}
\email{hanzhang3906@tsinghua.edu.cn}

\date{}

\maketitle

\begin{abstract}
         We investigate Birkhoff genericity on submanifolds of homogeneous space $X=SL_d(\bR)\ltimes (\bR^d)^k/ SL_d(\bZ)\ltimes (\bZ^d)^k$, where $d\geq 2$ and $k\geq 1$ are fixed integers. The submanifolds we consider are parameterized by unstable horospherical subgroup $U$ of a diagonal flow $a_t$ in $SL_d(\bR)$. As long as the intersection of the submanifold with any affine rational subspace has Lebesgue measure zero, we show that the trajectory of $a_t$ along Lebesgue almost every point on the submanifold gets equidistributed on $X$. This generalizes the previous work of Fr\k{a}czek, Shi and Ulcigrai in \cite{Shi_Ulcigrai_Genericity_on_curves_2018}.
     
Following the scheme developed by Dettmann, Marklof and Str\"{o}mbergsson in \cite{Marklof_Universal_hitting_time_2017}, we then deduce an application to universal hitting time statistics for integrable flows.
\end{abstract}


\section{Introduction}
Let $(X,\mathcal{B},\mu,R)$ be a probability measure preserving system, where $(X,\mathcal{B})$ is a Borel measurable space with probability measure $\mu$, and $R^t:X\to X$ is an $\bR$-action (or $\bZ$-action) preserving $\mu$. Assume that $R$ is ergodic, then Birkhoff's ergodic theorem (cf. \cite{Einsiedler_Ward_2011_book_MR2723325}) asserts that, for any $f\in L^1_{\mu}(X)$, 
\begin{align}\label{Birkhoff genericity for general measure preserving system}
    \frac{1}{T}\int_0^T f(R^t x)dt\xrightarrow{T\to \infty} \int_X f d\mu, \text{ (or } \frac{1}{N}\sum_{n=1}^N f(R^n x)\xrightarrow{N\to \infty}\int_X f d\mu ),
\end{align}
for $\mu$-almost every $x\in X$. In particular, (\ref{Birkhoff genericity for general measure preserving system}) holds for any $f\in C_c(X)$, where $C_c(X)$ denotes the collection of all compactly supported and continuous functions on $X$. Therefore, (\ref{Birkhoff genericity for general measure preserving system}) implies that for $\mu$-almost every $x\in X$, 
\begin{align}\label{Birkhoff genericity in weak* topology}
    \frac{1}{T}\int_0^T \delta_{R^t x} dt\xrightarrow{T\to \infty} \mu, \text{ (or }\frac{1}{N}\sum_{n=1}^N \delta_{R^n x}\xrightarrow{N\to \infty} \mu),
\end{align}
in the weak* topology on the set of all probability measures on $X$. Here $\delta$ is the Dirac measure on $X$.

For $x\in X$, we say that $x$ is \textbf{Birkhoff generic} with respect to $(\mu,R)$ if $x$ satisfies (\ref{Birkhoff genericity in weak* topology}). Given a {Radon} measure $\nu$ on $X$ (possibly singular to $\mu$), if $\nu$-almost every $x\in X$ is Birkhoff generic with respect to $(\mu,R)$, we say that $\nu$ is Birkhoff generic with respect to $(\mu,R)$. It is then natural to ask the following

\begin{question}\label{question 1}
Under what conditions the measure $\nu$ is Birkhoff generic with respect to $(\mu,R)$?
\end{question}

This question had been previously studied in the case of $X=\bR/\bZ$, $\mu_X$ is the Lebesgue measure on $X$ and $R^n=\times n \text{ mod } \bZ$ in \cite{Bernard_1995_Nombres}. It was shown that for any $m,n\in \mathbb{N}$, any $R^m$ invariant ergodic probability measure $\nu$ is Birkhoff generic with respect to $(R^n,\mu_X)$. This result was strengthened later in \cite{Hochman_Shmerkin_2015_Equidistribution_from_fractal_measures}. An analogous question was also studied in the context of moduli space of translation surfaces in \cite{Chaika_Eskin_2015_Every_flat_surface_is_Birkhoff_MR3395258}.

We consider Question \ref{question 1} in the setting of homogeneous dynamics. Let $X=G/\Gamma$, where $G$ is a Lie group and $\Gamma$ is a lattice in $G$. {Here and hereafter} let $\mu=\mu_X$ be the $G$-invariant probability measure on $X$. Let $\{R^t\}_{t\in \bR}$ be a one-parameter flow in $X$. Assume that $R^t=u(t)$, where $\{u(t)\}_{t\in \bR}$ is a one-parameter unipotent flow, that is, the adjoint action of $u(t)$ on the Lie algebra of $G$ is unipotent. In this case, Ratner's uniform distribution theorem \cite{ratner1991raghunathanduke} says that for any $x\in X$, $x$ is Birkhoff generic with respect to $(\nu,u(t))$, where $\nu$ is the $u(t)$-invariant probability measure supported on the orbit closure $\overline{\{u(t) x:t\in \bR\}}$. By Ratner's orbit closure theorem \cite{ratner1991raghunathanannals}, these orbit closures are all homogeneous. Thus this provides a satisfactory answer to Question \ref{question 1}.

On the other hand, when $R^t=a_t$, {here and hereafter} $\{a_t\}_{t\in \bR}$ is a one-parameter diagonal flow on $G$, that is, the adjoint action of $a_t$ on Lie algebra of $G$ is semisimple, a description of Birkhoff genericity of $x\in X$ under the flow $a_t$ is much harder. Indeed, the question of describing the orbit closures of diagonal action remains open (cf. \cite{Margulis_2000_problems_and_conjectures_MR1754775}\cite[Conjecture 1]{Tomanov_Actions_of_maximal_tori_2000}).

Nevertheless, there is a natural class of probability measures $\nu$ on $X$ that are interesting to study with respect to Question \ref{question 1}. This class of measures {is} given as follows. We define the unstable horospherical subgroup $U^+$ with respect to $a_t$ by
\begin{align*}
    U^+:=\{g\in G:a_{-t}g a_t\xrightarrow{t\to \infty} Id\},
\end{align*}
where $Id$ is the identity element of $G$. Let $Y\subset U^+$ be a submanifold and $\nu$ be {a normalized bounded supported volume measure of $Y$ (here and hereafter by normalized measure, we mean that $\nu$ is renormalized to be a probability measure).} It has been proved in \cite{Shah_Limitdistributionsofexpandingtranslatesofcertain_1996}\cite{shah2009equidistribution}\cite{shah2009limiting}\cite{yang2020equidistribution} that if $Y$ {satisfies certain algebraic conditions}, then the translation of the measure $\nu$ under $a_t$ converges weakly to $\mu$ as $t\to \infty$. That is, for any $f\in C_c(X)$,
\begin{align*}
    \int f(a_t x)d\nu(x)\xrightarrow{t\to \infty} \int f d\mu.
\end{align*}
As in Question \ref{question 1}, it is curious to ask the following 
\begin{question}\label{question 2}
{Assume that a bounded supported normalized volume measure} $\nu$ of a submanifold $Y\subset U^+$ is such that the translate of $\nu$ under $a_t$ is equidistributed with respect to $\mu$, is it true that $\nu$ is {also} Birkhoff generic with respect to $(\mu,a_t)$?
\end{question}

{Roughly speaking, Question \ref{question 2} is answered when the manifold $Y$ is considerably "large" compared to the unstable horosphere subgroup of $a_t$.} In \cite{Shi_Pointwiseequidistribution_2020}, Shi {considered} the situation where
$G$ is a semisimple Lie group, $Y=U$ is the $a_t$ expanding subgroup (cf. \cite{Shi_expanding_cone_2015}) of $U^+$ and $\nu$ {is a normalized bounded supported} Haar measure on $U$. {By \cite{Shah_Limitdistributionsofexpandingtranslatesofcertain_1996}, $\nu$ satisfies the assumption of Question \ref{question 2}. Shi showed that $\nu$ is also Birkhoff generic with respect to ($\mu,a_t$)}, and thus gave an affirmative answer to Question \ref{question 2}. In the special case where $G=SL_d(\bR)$ and $\Gamma=SL_d(\bZ)$, authors in \cite{Kleinbock_pointwiseequidistributionwithanerrorrate_2017} also obtained the effective convergence rate of (\ref{Birkhoff genericity in weak* topology}).

In \cite{Shi_Ulcigrai_Genericity_on_curves_2018}, the authors considered the setting {where} $G=SL_2(\bR)\ltimes \bR^2$, $\Gamma=SL_2(\bZ)\ltimes \bZ^2$ and $X=G/\Gamma$. {One of the main results in \cite{Shi_Ulcigrai_Genericity_on_curves_2018} asserts that} if $Y$ is a $C^1$ curve in $U^+$ that intersects any affine rational line in a {Lebesgue} null set, then {a normalized bounded supported volume} measure $\nu$ on $Y$ is Birkhoff generic with respect to $(\mu,a_t)$. By the equidistribution result of translation of such $\nu$ under $a_t$ in \cite{Marklof_Universal_hitting_time_2017}, 
{the result in \cite{Shi_Ulcigrai_Genericity_on_curves_2018}} also gives an affirmative answer to Question \ref{question 2} in the case {where} $Y$ is a curve.

{Let $X=SL_3(\bR)/SL_3(\bZ)$. In a recent preprint \cite[Theorem 1.4]{Kleinbock_Saxce_Shah_Yang_2021_Equidistribution_in_the_space_of}}, it is shown that when the natural measure $\nu$ on the planar line $L\subset U^+$ gets equidistributed under $a_t$, then for $\nu$ almost every point $x$ in $L$, the orbit $\{a_t x\}_{t\geq 0}$ is dense in $X$. This supports an affirmative answer to Question \ref{question 2}.

The aim of this paper is to generalize the genericity results in \cite{Shi_Ulcigrai_Genericity_on_curves_2018} to $X=SL_d(\bR)\ltimes (\bR^d)^k/ SL_d(\bZ)\ltimes (\bZ^d)^k$ and gives an affirmative answer to Question \ref{question 2} {when the manifold $Y$ satisfies certain diophantine condition.}
We also deduce an application of our results to the statistics of universal hitting time for integrable flows.

\subsection{Notations} From now on, any vector in the Euclidean space will be taken to be a column vector, and we will use boldface letters to denote vectors and matrices. 
Also, a.e. will be the shorthand for Lebesgue almost everywhere.
$|\cdot|$ will denote Lebesgue measure of measurable subsets of Euclidean space or absolute value of real numbers. $\norm{\cdot}$ will denote the standard Euclidean norm and $\norm{\cdot}_{\infty}$ the sup norm of a vector or matrix. {Throughout this article, for two matrices $A$ and $B$, $A \cdot B$  will denote \textbf{matrix multiplication.}}

For $m,n\in \mathbb{N}$, $Mat_{m\times n}(\bR)$ will denote the space of $m$ by $n$ real matrices. $(\bR^n)^m$ is the direct product of m copies of $\bR^n$. 

Fix integers $d\geq 2$, $k\geq 1$. Let $G\prm=SL_d(\bR)$, $G=SL_d(\bR)\ltimes (\bR^d)^k$, $\Gamma\prm=SL_d(\bZ)$ and $\Gamma=SL_d(\bZ)\ltimes (\bZ^d)^k$. It is well known that $\Gamma\prm$ is a lattice in $G\prm$ and $\Gamma$ is a lattice in $G$.

Let $X=G/\Gamma$. {Denote} $\mu_X$ the $G$-invariant probability measure on $G/\Gamma$. Note that the action of $G\prm$ on $(\bR^d)^k$ is given by 
\begin{align*}
    g\cdot \mathbf{v}=(g\cdot \mathbf{v_1},\cdots,g\cdot \mathbf{v_k}),
\end{align*}
where $g\in G\prm$ and $\mathbf{v}=(\mathbf{v_1},\cdots,\mathbf{v}_k)$ with $\mathbf{v_i}\in \bR^d$. The multiplication law in $G$ is given by
\begin{align*}
    (g,\mathbf{v})\cdot (g\prm,\mathbf{v\prm})=(g\cdot g\prm, \mathbf{v}+g\cdot \mathbf{v\prm}).
\end{align*}
$G\prm$ is naturally embedded into $G$ by
\begin{align*}
    G\prm \cong (G\prm,\mathbf{0})\le G.
\end{align*}

Fix an $r\in \{1,\cdots,d-1\}$. For $t\in \bR$ and $\mathbf{s}\in Mat_{r\times (d-r)}(\bR)$, denote 
\begin{align}
    &a_t=diag[e^{(d-r)t},\cdots,e^{(d-r)t},e^{-rt},\cdots,e^{-rt}],\label{definition of at}\\
&u(\mathbf{s})=
\begin{bmatrix}
 \mathbf{1}_r & \mathbf{s}\\
        \mathbf{0}_{d-r,r} & \mathbf{1}_{d-r}
\end{bmatrix},\\
& U:=\{u(\mathbf{s}):\mathbf{s}\in Mat_{r\times (d-r)}(\bR)\}\label{definition of U}\cong Mat_{r\times (d-r)}(\bR).
\end{align}

For a column vector or a matrix $\mathbf{v}$, let $(\mathbf{v})_{\leq r}$ (or $(\mathbf{v})_{>r})$ be the first $r$ rows (or last $d-r$ rows) of $\mathbf{v}$. For example, if $\mathbf{v}\in (\bR^d)^k$, then
\begin{align*}
    (\mathbf{v})_{\leq r}\in  (\bR^r)^k, \textbf{ }
    (\mathbf{v})_{>r}\in (\bR^{d-r})^k.
\end{align*}

With the above notations, the unstable horospherical subgroup $U^+$ of $a_t$ in $G$ is 
\begin{align*}
   U^+= U\cdot \{(Id,\begin{bmatrix}
       (\mathbf{v})_{\leq r}\\
       \mathbf{0}
   \end{bmatrix}):\mathbf{v}\in (\bR^d)^k\}.
\end{align*}

{Lastly,} for a map $\boldsymbol{\varphi}:Mat_{r\times(d-r)}(\bR)\to (\bR^d)^k$, we write $\boldsymbol{\varphi}(\mathbf{s})=(\varphi_{ij}(\mathbf{s}))_{1\leq i \leq d, 1\leq j \leq k}$. 

We also write 
\begin{align}\label{definition of u_varphi}
    u_{\boldsymbol{\varphi}}(\mathbf{s}):=u(\mathbf{s})\cdot (Id,\boldsymbol{\varphi}(\mathbf{s})).
\end{align}

\subsection{Main results}

For any $\mathbf{s}\in Mat_{r\times(d-r)}(\bR)$ and $T>0$, define the probability measure
\begin{align}\label{definition of mu s}
    \mu_{\mathbf{s},T}=\frac{1}{T}\int_0^T \delta_{a_t u_{\boldsymbol{\varphi}}(\mathbf{s})\Gamma} dt.
\end{align}
As before, we say that $u_{\boldsymbol{\varphi}}(\mathbf{s})\Gamma$ is Birkhoff generic with respect to $(X,\mu_X,a_t)$ if $\mu_{\mathbf{s},T}$ converges to $\mu_X$ {in the weak*-topology} as $T\to \infty$.

One of our main results is the following:
\begin{theorem}\label{genericity for special trajectories}
Let $\mathcal{U}\subset Mat_{r\times(d-r)}(\bR)$ be a bounded open subset. Let $\boldsymbol{\varphi}:\mathcal{U}\to (\bR^d)^k$ be a $C^1$-map satisfying $(\boldsymbol{\varphi}(\mathbf{s}))_{>r}\equiv \mathbf{0}$ for any $\mathbf{s}\in \mathcal{U}$. Assume that for any $\mathbf{m}\in \bZ^{k}\setminus \{\mathbf{0}\}$,
\begin{align}\label{genericity condition for special}
    |\{\mathbf{s}\in \mathcal{U}:(\boldsymbol{\varphi}(\mathbf{s}))_{\leq r}\cdot \mathbf{m}\in \mathbf{s} \cdot \bZ^{d-r}+\bZ^r\}|=0,
\end{align}
then for Lebesgue a.e. $\mathbf{s}\in \mathcal{U}$, $u_{\boldsymbol{\varphi}}(\mathbf{s})\Gamma$ is Birkhoff generic with respect to $(X,\mu_X,a_t)$. 
\end{theorem}

Using the observation that if a trajectory equidistributes with respect to $\mu_X$, then all other parallel trajectories will also equidistribute with respect to $\mu_X$ (see Lemma \ref{two asymptotic parallel trajectories converge to the same}), we can remove the assumption that $(\boldsymbol{\varphi}(\mathbf{s}))_{>r}\equiv \mathbf{0}$ and strengthen Theorem \ref{genericity for special trajectories} to the following

\begin{corollary}\label{corollary genericity for any c1 varphi}
Let $\mathcal{U}$ be a bounded open subset of $Mat_{r\times(d-r)}(\bR)$. Let $\boldsymbol{\varphi}:\mathcal{U}\to {(\mathbb{R}^d)^k}$ be a $C^1$ map.
If for any $\mathbf{m}\in \bZ^k\setminus\{\mathbf{0}\}$,
\begin{align*}
    |\{\mathbf{s}\in \mathcal{U}:(\boldsymbol{\varphi}(\mathbf{s}))_{\leq r} \cdot \mathbf{m}\in \mathbf{s} \cdot \bZ^{d-r}+\bZ^r\}|=0,
\end{align*}
then for Lebesgue a.e. $\mathbf{s}\in \mathcal{U}$, $u_{\boldsymbol{\varphi}}(\mathbf{s})\Gamma$ is Birkhoff generic with respect to $(X,\mu_X,a_t)$. 
\end{corollary}
We also obtain the following variants of Theorem \ref{genericity for special trajectories}.
\begin{theorem}\label{theorem:genericity for general trajectory}
Let $\mathcal{U}$ be a bounded open subset of $Mat_{r\times(d-r)}(\bR)$. Let $\boldsymbol{\varphi}:\mathcal{U}\to (\bR^d)^k$ be a $C^1$ map. Let $\mathbf{M}\in SL_d(\bR)$. If for any $\mathbf{m}\in \bZ^{k}\setminus\{\mathbf{0}\}$,
\begin{align}\label{genericity condition for general trajectory}
    |\{\mathbf{s}\in \mathcal{U}: \boldsymbol{\varphi}(\mathbf{s})\cdot \mathbf{m}\in \mathbf{M}^{-1}u(-\mathbf{s})\cdot \begin{bmatrix}
        \mathbf{0}\\
        \bR^{d-r}
    \end{bmatrix}
    +{\mathbb{Z}^d}\}|=0.
\end{align}
Then for Lebesgue a.e. $\mathbf{s}\in \mathcal{U}$, $u(\mathbf{s})\mathbf{M}(Id,\boldsymbol{\varphi}(\mathbf{s}))\Gamma$ is Birkhoff generic with respect to $(X,\mu_X,a_t)$.
\end{theorem}
\begin{remark}
By the equidistribution result in \cite{Marklof_Universal_hitting_time_2017}, Theorem \ref{theorem:genericity for general trajectory} gives an affirmative answer to Question \ref{question 2}.
\end{remark}
\begin{remark}\label{Remark: necessity of conditions on Theorem 1.3,1.5}
{The conditions (\ref{genericity condition for special}) and (\ref{genericity condition for general trajectory}) are indeed necessary. For example in Theorem \ref{genericity for special trajectories}, suppose that $(\boldsymbol{\varphi}(\mathbf{s}))_{\leq r}\cdot \mathbf{m}\in \mathbf{s} \cdot \bZ^{d-r}+\bZ^r$ for some $\mathbf{s}\in \mathcal{U}$ and $\mathbf{m}\in \mathbb{Z}^k\setminus \{\mathbf{0}\}$, then the $a_t$ trajectory along $u_{\boldsymbol{\varphi}}(\mathbf{s})\Gamma$ will concentrate on a proper submanifold of $G/\Gamma$. }
\end{remark}
\begin{corollary}\label{Corollary: genericity for general base point after u_varphi}
{Let $\mathcal{U}$ be a bounded open subset of $Mat_{r\times(d-r)}(\bR)$. Let $\boldsymbol{\varphi}:\mathcal{U}\to (\bR^d)^k$ be a $C^1$ map. Let $(\mathbf{M},\mathbf{v})\in G$. If for any $\mathbf{m}\in \mathbb{Z}^k\setminus \{\mathbf{0}\}$,
\begin{align*}
    |\{\mathbf{s}\in \mathcal{U}: (\boldsymbol{\varphi}(\mathbf{s})+\mathbf{v}) \cdot \mathbf{m}\in u(-\mathbf{s})\cdot \begin{bmatrix}
        \mathbf{0}\\
        \bR^{d-r}
    \end{bmatrix}
    +\mathbf{M}\cdot \mathbb{Z}^d\}|=0,
\end{align*}
then for Lebesgue a.e. $\mathbf{s}\in \mathcal{U}$, $u_{\boldsymbol{\varphi}}(\mathbf{s})(\mathbf{M},\mathbf{v})\Gamma$ is Birkhoff generic with respect to $(X,\mu_X,a_t)$.
}
\end{corollary}
For $1\leq i\leq d$, let $\mathbf{e}_i\in \bR^d$ be the column vector such that $i$-th row of $\mathbf{e}_i$ is $1$ and others are $0$. 
\begin{corollary}\label{genericity on orbit of maximal compact group}
Let $\mathcal{U}$ be a bounded open subset of $Mat_{r\times(d-r)}(\bR)$. Let $\mathbf{E_1}:\mathcal{U}\to SO_d(\bR)$ be a smooth map such that the map $\mathbf{s}\mapsto \mathbf{E_1}(\mathbf{s})^{-1}\cdot [\mathbf{e}_{r+1},\cdots,\mathbf{e}_d]$ has a nonsingular differential at Lebesgue almost every $\mathbf{s}\in \mathcal{U}$. Let $\boldsymbol{\varphi}:\mathcal{U}\to (\bR^d)^k$ be a $C^1$ map. Assume that for any $\mathbf{m}\in \bZ^{k}\setminus\{\mathbf{0}\}$,
\begin{align*}
    |\{\mathbf{s}\in \mathcal{U}:\boldsymbol{\varphi}(\mathbf{s})\cdot\mathbf{m}\in \mathbf{E}_1(\mathbf{s})^{-1}\cdot \begin{bmatrix}
        \mathbf{0}\\
        \bR^{d-r}
    \end{bmatrix}
    +{\mathbb{Z}^d} \}|=0,
\end{align*}
then for Lebesgue a.e. $\mathbf{s}\in \mathcal{U}$, $\mathbf{E}_1(\mathbf{s})(Id,\boldsymbol{\varphi}(\mathbf{s}))\Gamma$ is Birkhoff generic with respect to $(X,\mu_X,a_t)$.
\end{corollary}

Corollary \ref{genericity on orbit of maximal compact group} will allow us to deduce an application to universal hitting time for integrable flows in $d$-torus $\mathbb{T}^d$ (see Theorem \ref{theorem on universal hitting time}).

\subsection{Ingredients of the proof}
Proof of Theorem \ref{genericity for special trajectories} follows the similar strategy as in \cite{Shi_Ulcigrai_Genericity_on_curves_2018}. However, some new ingredients are required. We need the description of orbit closures of $G\prm$ in $X$. This can be done using Ratner's orbit closure theorem following the approach in \cite{Marklof_Universal_hitting_time_2017}. 

We need to construct a suitable mixed height function in our situation, which measures the distance of point to the cusp and singular submanifolds. 

Also due to higher rank, some technical difficulties arise in the proof of uniform contraction property of the mixed height function. To overcome these difficulties, we apply a linear algebra lemma (see Lemma \ref{basic linear algebra lemma}) which is inspired by the proof of \cite[Proposition 3.4]{kleinbock1998flows}.

\subsection{Overview}
In Section \ref{section Reductions and proof}, we make some reductions and give a proof of Theorem \ref{genericity for special trajectories} and Corollary \ref{corollary genericity for any c1 varphi}.

In Section \ref{Section Orbit closure}, we will investigate the orbit closure of $G\prm$ in $X$ using Ratner's orbit closure theorem. 

In Section \ref{section unipotent invariance for a.e. s}, we prove that for a.e. $\mathbf{s}\in \mathcal{U}$, the limit measure is invariant under the unipotent group $U$. This enables us to apply Ratner's measure classification theorem.

In Section \ref{section Margulis' height function} and Section \ref{section mixed height functions}, we will construct mixed height function $\beta_{\mathbf{m}}$ for $\mathbf{m}\in \bZ^k\setminus \{\mathbf{0}\}$, and {give a proof of} its uniform contraction property.

In Section \ref{section proof of main proposition}, we prove Proposition \ref{limit measure spend very little measure on singular sets} using mixed height function.

In Section \ref{section: more general version of theorem genericity condition for special}, we deduce variants of Theorem \ref{genericity for special trajectories}.

In Section \ref{section: application to universal hitting time}, we deduce an application to universal hitting time statistics.

\section{Reductions and proof of Theorem \ref{genericity for special trajectories}}\label{section Reductions and proof}

In this section, assuming several Theorems/Propositions/Lemmas that will be proved later, we give a proof of Theorem \ref{genericity for special trajectories}.

By Proposition \ref{unipotent invariance for a.e. s}, for a.e. $\mathbf{s}\in \mathcal{U}$, {after possibly} passing to a subsequence the weak* limit $\mu_{\mathbf{s}}$ of $\mu_{\mathbf{s},T}$ is $U$-invariant.
From the definition of $\mu_{\mathbf{s},T}$ (see (\ref{definition of mu s})), it follows that $\mu_{\mathbf{s}}$ is also $D=\{a_t:t\in \bR\}$-invariant.
Hence for a.e. $\mathbf{s}\in \mathcal{U}$, $\mu_{\mathbf{s}}$ is $DU$-invariant.
Note that $DU$ is an epimorphic subgroup of $G\prm=SL_d(\bR)$. By \cite{Mozes_Epimorphic_subgroups_1995}, {as} $\mu_{\mathbf{s}}$ is a probability measure invariant under $DU$, $\mu_{\mathbf{s}}$ is $G\prm$-invariant. By Ratner's measure classification theorem, any $G\prm$ invariant and ergodic probability measure is supported on an orbit {closure} of $G\prm$ on $X$.

A consequence of Ratner's orbit closure theorem ({Theorem} \ref{orbit closure of certain base point}) shows that any orbit closure of $G\prm$ is either

(1) the whole $X$, or

(2) {concentrated} in a proper closed submanifold $X_{\mathbf{m}}$ for some $\mathbf{m}\in \bZ^k\setminus \{\mathbf{0}\}$, where \begin{align*}
    X_{\mathbf{m}}=\{(g,g\mathbf{v})\Gamma:g\in G\prm,\mathbf{v}\cdot\mathbf{m}\in \bZ^d\}.
\end{align*}

Therefore, it remains to show that for a.e. $\mathbf{s}\in \mathcal{U}$, $\mu_{\mathbf{s}}$ is a probability measure on $X$ and $\mu_{\mathbf{s}}(X_{\mathbf{m}})=0$ for any $\mathbf{m}\in \bZ^k\setminus \{\mathbf{0}\}$.

Let 
\begin{align}\label{definition of M1}
    M_1:=N_1 \cdot (\max_{1\leq i \leq r, 1\leq j \leq d-r}\sup_{s\in \mathcal{U}}\norm{\partial_{ij}\boldsymbol{\varphi}(\mathbf{s})}_{\infty})+1,
\end{align}
{where $N_1=8r^2 k^{1/2}(d-r)$, and $\partial_{ij} \boldsymbol{\varphi}$ is a $d$ by $k$ matrix whose $(p,q)$-th entry is $\partial \varphi_{pq}/ \partial s_{ij}$. Here the choice of $N_1$ is flexible, we just choose a value for $N_1$ that is convenient for us.}

By assumption (\ref{genericity condition for special}) of Theorem \ref{genericity for special trajectories}, for any $\mathbf{m}\in \bZ^k\setminus \{\mathbf{0}\}$, the set
\begin{align}\label{genericity condition under regularity}
   Bad_{\mathbf{m}}=\{\mathbf{s}\in \mathcal{U}:\exists \mathbf{a}\in \bZ^{d-r}, \mathbf{b}\in \bZ^r \text{ such that } (\boldsymbol{\varphi}(\mathbf{s}))_{\leq r}\boldsymbol{\cdot }\mathbf{m}= \mathbf{s}
    \boldsymbol{\cdot} \mathbf{a}+\mathbf{b},\nonumber \\ \text{ and } \norm{\mathbf{a}}_{\infty}\leq M_1 \norm{\mathbf{m}}
    \}
\end{align}
has Lebesgue measure zero.

Since there are only finitely many $\mathbf{a}\in \bZ^{d-r}$ such that $\norm{\mathbf{a}}_{\infty}\leq M_1\cdot \norm{\mathbf{m}}$, $Bad_{\mathbf{m}}$ is a closed set with Lebesgue measure zero. Thus to prove Theorem \ref{genericity for special trajectories}, it suffices to prove it for a closed cube contained in $\mathcal{U}\setminus Bad_{\mathbf{m}}$. {Now let's fix a closed cube $I\subset \mathcal{U}\setminus Bad_{\boldsymbol{m}}$.}

Let $K$ be a measurable subset of $X$. For any $T>0$, we define the average operator $\mathcal{A}_K^T:\mathcal{U}\to [0,1]$ by 
\begin{align*}
    \mathcal{A}_K^{T}(\mathbf{s})=\frac{1}{T}\int_0^T \chi_K(a_t u_{\boldsymbol{\varphi}}(\mathbf{s})\Gamma)dt=\mu_{\mathbf{s},T}(K),
\end{align*}
where $\chi_K$ is the characteristic function of $K$. 

The key proposition, which ensures that $\mu_{\mathbf{s}}$ is a probability measure putting zero mass on $X_{\mathbf{m}}$ for a.e. $\mathbf{s}\in I$, is the following:

\begin{proposition}\label{limit measure spend very little measure on singular sets}
Let $\mathbf{m}\in \bZ^k\setminus \{\mathbf{0}\}$. Let $\boldsymbol{\varphi}:I\to (\bR^d)^k$ be a $C^1$ map satisfying $(\boldsymbol{\varphi}(\mathbf{s}))_{>r}\equiv \mathbf{0}$ for any $\mathbf{s}\in I$. Suppose that 
\begin{align}\label{consequence of regularity condition}
      \inf_{\mathbf{s}\in I}\{\norm{(\boldsymbol{\varphi}(\mathbf{s}))_{\leq r}\boldsymbol{\cdot }\mathbf{m}-\mathbf{s}
    \boldsymbol{\cdot} \mathbf{a}-\mathbf{b}}_{\infty}
    : \norm{\mathbf{a}}_{\infty}\leq M_1 \norm{\mathbf{m}}, \mathbf{a}\in \bZ^{d-r},\mathbf{b}\in \bZ^r
    \}>0.
\end{align}
Then for any $\epsilon>0$, there exists a compact subset ${K_{\epsilon}}\subset X\setminus X_{\mathbf{m}}$ and $\mathfrak{v}>0$ such that for any $T>0$,
\begin{align}\label{exponential decay of average operator}
    |\{\mathbf{s}\in I: {\mathcal{A}_{K_{\epsilon}}^T}(\mathbf{s})\leq 1-\epsilon \}|\leq e^{-\mathfrak{v}T}|I|.
\end{align}
\end{proposition}

It will be proved in Lemma \ref{proof of consequence of regularity condition} that condition (\ref{consequence of regularity condition}) in Proposition \ref{limit measure spend very little measure on singular sets} follows from condition (\ref{genericity condition for special}) in Theorem \ref{genericity for special trajectories}.

Proposition \ref{limit measure spend very little measure on singular sets} will be proved in Section \ref{section mixed height functions}.
Combining Borel-Cantelli lemma, a direct consequence of Proposition \ref{limit measure spend very little measure on singular sets} is the following:

\begin{proposition}\label{limit measure spend zero measure on singular sets}
Under the assumptions of Theorem \ref{genericity for special trajectories}, for a.e. $\mathbf{s}\in \mathcal{U}$, by possibly passing to a subsequence, { $\mu_{\mathbf{s},T}$ converges to a probability measure $\mu_{\mathbf{s}}$ on $X$ in weak*-topology as $T\to \infty$, and $\mu_{\mathbf{s}}(X_{\mathbf{m}})=0$ for any $\mathbf{m}\in \bZ^k\setminus \{\mathbf{0}\}$.}
\end{proposition}

\begin{proof}
Fix an $\mathbf{m}\in \bZ^k\setminus \{\mathbf{0}\}$ and $\epsilon>0$. By Proposition \ref{limit measure spend very little measure on singular sets}, we can choose a compact subset $K_{\epsilon}$ of $X\setminus X_{\mathbf{m}}$ such that (\ref{exponential decay of average operator}) holds for any $T>0$. Let $T=n\in \mathbb{N}$ and apply Borel-Cantelli lemma to the collection of the sets
\begin{align*}
    \{\mathbf{s}\in I:\mathcal{A}_{K_{\epsilon}}^n(\mathbf{s})\leq 1-\epsilon\},  n\in \mathbb{N}.
\end{align*}
We can find a measurable subset $I_{\mathbf{m}}^{\epsilon}$ of $I$ with full measure such that for any $\mathbf{s}\in I_{\mathbf{m}}^{\epsilon}$, $\mathcal{A}_{K_{\epsilon}}^n(\mathbf{s})> 1-\epsilon$ for {all sufficiently large} $n\in \mathbb{N}$.
Therefore, for any $\mathbf{s}\in I_{\mathbf{m}}^{\epsilon}$, $\mu_{\mathbf{s}}(X)\geq 1-\epsilon$ and $\mu_{\mathbf{s}}(X_{\mathbf{m}})\leq \epsilon$. Let $I_{\mathbf{m}}=\cap_{n=1}^{\infty}I_{\mathbf{m}}^{\frac{1}{n}}$, then $I_{\mathbf{m}}$ has full Lebesgue measure in $I$, and for any $\mathbf{s}\in I_{\mathbf{m}}$, $\mu_{\mathbf{s}}(X)=1$ and $\mu_{\mathbf{s}}(X_{\mathbf{m}})=0$.

To complete the proof, we let $I\prm=\bigcap_{\mathbf{m}\in \bZ^k\setminus\{\mathbf{0}\}}I_{\mathbf{m}}$. Then $I\prm$ has full Lebesgue measure in $I$ and the proposition holds for all $\mathbf{s}\in I\prm$.
\end{proof}

Assuming Proposition \ref{limit measure spend zero measure on singular sets}, Proposition \ref{unipotent invariance for a.e. s} and Theorem \ref{orbit closure of certain base point}, we are ready to prove Theorem \ref{genericity for special trajectories}:

\begin{proof}[Proof of Theorem \ref{genericity for special trajectories}]
By Proposition \ref{limit measure spend zero measure on singular sets} and Proposition \ref{unipotent invariance for a.e. s}, we conclude that for a.e. $\mathbf{s}\in \mathcal{U}$, the weak* limit $\mu_{\mathbf{s}}$ of $\mu_{\mathbf{s},T}$ as $T\to \infty$ is

(1) a probability measure on $X$, and $\mu_{\mathbf{s}}(X_{\mathbf{m}})=0$ for any $\mathbf{m}\in \bZ^k\setminus \{\mathbf{0}\}$;

(2) $DU$-invariant.

Since $DU$ is an epimorphic subgroup of $G\prm=SL_d(\bR)$, and $\mu_{\mathbf{s}}$ is a $DU$-invariant probability measure on $X$, $\mu_{\mathbf{s}}$ is $G\prm$-invariant  by \cite[Theorem 1]{Mozes_Epimorphic_subgroups_1995}.

By Ratner's measure classification theorem \cite{ratner1991raghunathanannals}, any ergodic component of such $\mu_{\mathbf{s}}$ is supported on an orbit closure of $G\prm$ on $X$. Theorem \ref{orbit closure of certain base point} describes all the possible orbit closures of $G\prm$ on $X$: either it is $X$ or it is concentrated on $X_{\mathbf{m}}$ for some $\mathbf{m}\in \bZ^k\setminus \{\mathbf{0}\}$.

Since $\mu_{\mathbf{s}}(X_{\mathbf{m}})=0$ for any $\mathbf{m}\in \bZ^k\setminus \{\mathbf{0}\}$, we conclude that for a.e. $\mathbf{s}\in \mathcal{U}$, $\mu_{\mathbf{s}}=\mu_X$.
\end{proof}
We note the following
\begin{lemma}\label{two asymptotic parallel trajectories converge to the same}
Assume that for some $x\in X=G/\Gamma$,  
\begin{align*}
    \frac{1}{T}\int_0^T \delta_{a_t x}dt\xrightarrow{T\to \infty}\mu_{G/\Gamma}, \text{ in weak*-topology},
    \end{align*}
and for some $g\in G$, $a_t g a_{-t}\to Id\in G$ as $t\to \infty$, then
\begin{align*}
    \frac{1}{T}\int_0^T \delta_{a_t gx}dt\xrightarrow{T\to\infty}\mu_{G/\Gamma}, \text{ in weak*-topology}.
\end{align*}
\end{lemma}
For any $\boldsymbol{\varphi}:Mat_{r\times(d-r)}(\bR)\to (\bR^d)^k$, and any $\mathbf{s}\in Mat_{r\times(d-r)}(\bR)$, we can write
\begin{align*}
    a_t u_{\boldsymbol{\varphi}}(\mathbf{s})&= a_t (u(\mathbf{s}),\boldsymbol{\varphi}(\mathbf{s}))\\&=a_t(Id,\begin{bmatrix}
        \mathbf{0}\\
        (\boldsymbol{\varphi}(\mathbf{s}))_{>r}
    \end{bmatrix}
    )a_{-t}\cdot a_t (u(\mathbf{s}),\begin{bmatrix}
        (\boldsymbol{\varphi}(\mathbf{s}))_{\leq r}\\
        \mathbf{0}
    \end{bmatrix}
    ),
\end{align*}
where
\begin{align*}
    \boldsymbol{\varphi}(\mathbf{s})=\begin{bmatrix}
        \mathbf{0}\\
        (\boldsymbol{\varphi}(\mathbf{s}))_{>r}
        \end{bmatrix}+\begin{bmatrix}
        (\boldsymbol{\varphi}(\mathbf{s}))_{\leq r}\\
        \mathbf{0}
    \end{bmatrix}.
\end{align*}
Since
\begin{align*}
    a_t(Id,\begin{bmatrix}
        \mathbf{0}\\
        (\boldsymbol{\varphi}(\mathbf{s}))_{>r}
    \end{bmatrix})a_{-t}\to (Id,\mathbf{0}),
\end{align*}
by Lemma \ref{two asymptotic parallel trajectories converge to the same} and Theorem \ref{genericity for special trajectories}, Corollary \ref{corollary genericity for any c1 varphi} is proven.

\section{Orbit closure}\label{Section Orbit closure}
In this section, we will classify all orbit closures of $G\prm$ in $X$ following \cite{Marklof_Universal_hitting_time_2017}. Recall that $G\prm=SL_d(\bR)$ and $G=SL_d(\bR)\ltimes (\bR^d)^k$.

Consider a base point $(Id,\boldsymbol{\xi})\in G$. Since $G\prm$ is a simple Lie group, an application of Ratner's orbit closure theorem gives the following theorem describing the orbit closure of $G\prm\cdot(Id,\boldsymbol{\xi})\Gamma/\Gamma$ in $G/\Gamma$:

\begin{theorem}\label{orbit closure of certain base point}
The orbit closure $\overline{G\prm\cdot(Id,\boldsymbol{\xi})\Gamma/\Gamma}$ is $G/\Gamma$ if and only if for any $\mathbf{m}\in \bZ^k\setminus \{\mathbf{0}\}$, $\boldsymbol{\xi}\cdot \mathbf{m}\notin \bZ^d$.
\end{theorem}

By Ratner's orbit closure theorem (\cite{ratner1991raghunathanduke}), for any $\boldsymbol{\xi}\in (\bR^d)^k$, there exists a closed subgroup $H$ of $G$ containing $G\prm$ such that
\begin{align*}
    \overline{G\prm\cdot(Id,\boldsymbol{\xi})\Gamma/\Gamma}=H\cdot(Id,\boldsymbol{\xi})\Gamma/\Gamma,
\end{align*}
and $H\cdot(Id,\boldsymbol{\xi})\Gamma/\Gamma$ admits an $H$-invariant probability measure.

It can be checked that if there exists $\mathbf{m}\in \bZ^k\setminus\{\mathbf{0}\}$ such that $\boldsymbol{\xi}\cdot \mathbf{m}\in \bZ^d$, then
    $G\prm\cdot(Id,\boldsymbol{\xi})\Gamma\subset X_{\mathbf{m}}$,
where
\begin{align}\label{singular set of level m}
    X_{\mathbf{m}}=\{(g,g\mathbf{v})\Gamma:g\in G\prm,\mathbf{v}\in(\bR^d)^k \text{ such that } \mathbf{v}\cdot\mathbf{m}\in \bZ^d\}.
\end{align}
This is a closed submanifold of $X$ of codimension $d$. In this case, the orbit $G\prm\cdot(Id,\boldsymbol{\xi})\Gamma/\Gamma$ does not equidistribute in $X$.

The converse of Theorem \ref{orbit closure of certain base point} will follow from Lemmas \ref{description of intermediate subgroups}-\ref{consequence of assumption on xi}. We will follow the proof strategy of \cite[Theorem 3]{Marklof_Universal_hitting_time_2017}. Let $\boldsymbol{\xi}$, $H$ be as above. 
\begin{lemma}\label{description of intermediate subgroups}
There is a linear subspace $\boldsymbol{U}\subset \bR^k$ such that $H=SL_d(\bR)\ltimes L(\boldsymbol{U})$, where $L(\boldsymbol{U})$ {is a subset of $Mat_{d\times k}(\mathbb{R})$ such that for any element $\mathbf{v}$ of $L(\boldsymbol{U})$, each row vector of $\mathbf{v}$ is a vector in $\boldsymbol{U}$.}
\end{lemma}

\begin{proof}
Let $\boldsymbol{L}=\{\mathbf{v}\in (\bR^d)^k: (Id,\mathbf{v})\in H\}$. {Because $G\prm\subset H$, for any $\mathbf{v}\in \boldsymbol{L}$, we have $(g,\mathbf{0})\cdot (Id,\mathbf{v})\cdot (g,\mathbf{0})^{-1}=(Id,g\mathbf{v})\in H.$ It follows that $\boldsymbol{L}$ is $G\prm$-invariant and $SL_d(\mathbb{R})\ltimes \boldsymbol{L}\subset H$. For any $(g,\mathbf{v})\in H$, we have $(g^{-1},\mathbf{0})\cdot (g,\mathbf{v})=(Id,g^{-1}\mathbf{v})\in H$, so $g^{-1}\mathbf{v}\in \boldsymbol{L}$. Since $\boldsymbol{L}$ is $G\prm$-invariant, $\mathbf{v}\in \boldsymbol{L}$. Therefore $H= SL_d(\mathbb{R})\ltimes \boldsymbol{L}$.}

Let $A\in \mathfrak{sl}_d=Lie(G\prm)$, then for any $t\in \bR$, and any $\mathbf{v}\in \boldsymbol{L}$
\begin{align*}
    \frac{exp(tA)\mathbf{v}-\mathbf{v}}{t}\in \boldsymbol{L}.
\end{align*}
Let $t\to 0$, we obtain $A\cdot \mathbf{v}\in \boldsymbol{L}$. Recall that $\mathfrak{sl}_d$ consists of all trace zero $d\times d$ matrices.
Let $\mathbf{E}_{ij}$ be the $d\times d$ matrix with $1$ in the $(i,j)$-th entry and zero for all other entries.
Then for any $i\neq j$, $\mathbf{E}_{ij}\mathbf{v}\in \boldsymbol{L}$. Since $\mathbf{E}_{ij}\cdot \mathbf{E}_{ji}=\mathbf{E}_{ii}$, for any $i$ we have $\mathbf{E}_{ii}\mathbf{v}\in \boldsymbol{L}$ as well. Therefore, $\boldsymbol{L}$ is invariant under left multiplication of all $d\times d$ real matrices.
Since left multiplication is row operation, there is a linear subspace $\boldsymbol{U}\subset \bR^k$ such that $\boldsymbol{L}=L(\boldsymbol{U})$. 
\end{proof}

{Let $\pi_1:G\to G\prm$ be the natural projection map and $\Gamma_{L}=L(\boldsymbol{U})\cap \Gamma.$}

\begin{lemma}\label{property of intermediate subgroups}
Let $\boldsymbol{U}$ be the linear subspace of $\bR^k$ obtained by the Lemma \ref{description of intermediate subgroups}. Then $\boldsymbol{U}\cap \bZ^k$ is a lattice in $\boldsymbol{U}$ and $\boldsymbol{\xi}\in (\mathbb{Q}^d)^k+L(\boldsymbol{U})$.
\end{lemma}

\begin{proof}
By Lemma \ref{description of intermediate subgroups}, $H=SL_d(\bR)\ltimes L(\boldsymbol{U})$ and $H\cdot(Id,\boldsymbol{\xi})\Gamma/\Gamma$ is closed and admits an $H$-invariant probability measure, therefore $\Gamma_H=(Id,\boldsymbol{\xi})\Gamma(Id,-\boldsymbol{\xi})\cap H$ is a lattice in $H$. 

By \cite[Corollary 8.28]{Raghunathan_Discrete_subgroups_of_Lie_groups_1972}, $\Gamma_{L}$ is a lattice in $L(\boldsymbol{U})$, that is, $(\bZ^d)^k\cap L(\boldsymbol{U})$ is a lattice in $L(\boldsymbol{U})$. Thus $\boldsymbol{U}$ has a basis belonging to $\bZ^k$, and it follows that $\bZ^k\cap \boldsymbol{U}$ is a lattice in $\boldsymbol{U}$.

Recall that $\Gamma\prm=SL_d(\bZ)$. Now consider $\pi_1(\Gamma_H)=\{\gamma\in \Gamma\prm: \boldsymbol{\xi}-\gamma\cdot \boldsymbol{\xi}\in (\bZ^d)^k+L(\boldsymbol{U})\}$. Again by \cite[Corollary 8.28]{Raghunathan_Discrete_subgroups_of_Lie_groups_1972}, $\pi_1(\Gamma_H)$ is a lattice in $G\prm$. Therefore $\pi_1(\Gamma_H)$ is a finite index subgroup of $\Gamma\prm$. Pick a $\gamma\in \pi_1(\Gamma_H)$ such that $Id-\gamma$ is invertible, then $\boldsymbol{\xi}\in (\mathbb{Q}^d)^k+L(\boldsymbol{U})$.
\end{proof}

\begin{lemma}\label{consequence of assumption on xi}
Let $\boldsymbol{U}$ be the linear subspace of $\bR^k$ obtained by Lemma \ref{description of intermediate subgroups}. If for any $\mathbf{m}\in \bZ^k\setminus \{\mathbf{0}\}$, $\boldsymbol{\xi}\cdot \mathbf{m}\notin \bZ^d$. Then $\boldsymbol{U}=\bR^k$ and hence, $H=G$.
\end{lemma}
\begin{proof}
Suppose $\boldsymbol{U}\neq \bR^k$, then $\dim \boldsymbol{U}<k$. Since $\boldsymbol{U}\cap \bZ^k$ is a lattice in $\boldsymbol{U}$, there exists a nonzero $\mathbf{v}\in \bZ^k\cap \boldsymbol{U}^{\perp}$. Since $\boldsymbol{\xi}\cdot \mathbf{v}\in (\mathbb{Q}^d)^k\cdot \mathbf{v}+L(\boldsymbol{U})\cdot\mathbf{v}=(\mathbb{Q}^d)^k\cdot \mathbf{v}$, we can choose $\mathbf{m}$ to be a suitable integral multiple of $\mathbf{v}$ such that $\boldsymbol{\xi}\cdot \mathbf{m}\in \bZ^d$, this contradicts to the assumption of the lemma.
\end{proof}

\section{Unipotent invariance}\label{section unipotent invariance for a.e. s}
The collection of all probability measures on the one point compactification $X^*$ of $X$ is a compact space in weak*-topology. Therefore, for any $\mathbf{s}\in \mathcal{U}$, after possibly passing to a subsequence, we have
\begin{align*}
    \frac{1}{T}\int_0^T \delta_{a_tu_{\boldsymbol{\varphi}}(\mathbf{s})\Gamma}dt \xrightarrow{T\to \infty}\mu_{\mathbf{s}} \text{ in weak* topology},
\end{align*}
for some probability measure $\mu_{\mathbf{s}}$ on $X^*$. { Throughout this section, the function $\boldsymbol{\varphi}$ is assumed to be $C^1$ and satisfy $(\boldsymbol{\varphi}(\mathbf{s}))_{>r}\equiv \mathbf{0}$.}

\begin{proposition}\label{unipotent invariance for a.e. s}
For a.e. $\mathbf{s}\in \mathcal{U}$, $\mu_{\mathbf{s}}$ is $U$-invariant.
\end{proposition}

\begin{proof}
Since $\mathcal{U}$ is a bounded open subset of $Mat_{r\times(d-r)}(\bR)$, it is enough to prove the proposition for a.e. $\mathbf{s}$ in an open cube of $\mathcal{U}$.

We may choose an open interval $\mathbb{I}\subset \bR$ such that $\mathbb{I}^{r(d-r)}\subset \mathcal{U}$. For $1\leq i\leq r,1\leq j\leq d-r$, let $\mathbf{E}_{ij}\in Mat_{r\times(d-r)}(\bR)$ be the matrix with $1$ in $(i,j)$-th entry and zero otherwise.

If $s_1,s_2$ are two real numbers linearly independent over $\mathbb{Q}$, then the closure of the subgroup generated by $\{u(s_1\mathbf{E}_{ij}),u(s_2\mathbf{E}_{ij}):1\leq i \leq r, 1\leq j\leq d-r\}$ is $U$.

Therefore, given $s\prm \in \bR$, without loss of generality, it suffices to prove that for a.e. $\mathbf{s}\in \mathbb{I}^{r(d-r)}$, the limit measure $\mu_{\mathbf{s}}$ is invariant under $u(s\prm \mathbf{E}_{11})$.

Note that there exists a countable dense subset of $C_c(G/\Gamma)$ consisting of smooth functions.
Let $\psi\in C_c^{\infty}(G/\Gamma)$. For $t>0$ and $\mathbf{w}\in Mat_{r\times(d-r)}(\bR)$, define
\begin{align*}
    \psi_t(\mathbf{w})=\psi(a_tu_{\boldsymbol{\varphi}}(\mathbf{w})\Gamma)-\psi(u(s\prm \mathbf{E}_{11})a_tu_{\boldsymbol{\varphi}}(\mathbf{w})\Gamma).
\end{align*}
Hence, we only need to show that for this $\psi$, for a.e. $\mathbf{w}\in \mathbb{I}^{r(d-r)}$,
\begin{align*}
    \frac{1}{T}\int_0^T \psi_t(\mathbf{w})dt\xrightarrow{T\to \infty} 0.
\end{align*}
{This follows from Theorem \ref{Theorem: Kleinbock-Shi-Weiss theorem} and Lemma \ref{Lemma: correlation} as follows.}
\end{proof}
{
\begin{theorem}\cite[Theorem 3.1]{Kleinbock_pointwiseequidistributionwithanerrorrate_2017}\label{Theorem: Kleinbock-Shi-Weiss theorem}
Let $(Y,\mu)$ be a probability space. Let $F:Y\times \bR^+\to \bR$ be a bounded measurable function. Suppose that there exist $\delta>0$ and $c>0$ such that for any $l\geq t\geq 0$,
\begin{align}\label{high correlation condition}
    |\int_Y F(x,t)F(x,l)d\mu(x)|\leq c\cdot e^{-\delta\min(t,l-t)},
\end{align}
then given any $\epsilon>0$, for $\mu$-a.e. $y\in Y$, 
\begin{align*}
    \frac{1}{T}\int_0^T F(y,t)dt=o(T^{-\frac{1}{2}}\cdot log^{\frac{2}{3}+\epsilon}T).
\end{align*}
\end{theorem}
}
{
\begin{lemma}\label{Lemma: correlation}\label{high correlation}
There exist $c>0$ such that for any $t,l>0$,
\begin{align*}
    |\int_{\mathbb{I}^{r(d-r)}}\psi_t(\mathbf{w})\psi_l(\mathbf{w})d\mathbf{w}|\leq c\cdot e^{-|l-t|}.
\end{align*}
\end{lemma}
}
\begin{proof}
{
In the following proof, for positive valued functions $f,g$, we write $f=O(g)$ if there exists a positive constant $C$ depending only on $\psi,\boldsymbol{\varphi},s\prm$ and $|\mathbb{I}|$ (these are fixed throughout the proof) such that $f\leq C g$. Also, for any positive number $\epsilon>0$, we let $O_G(\epsilon)$ denote a group element in a $ O(\epsilon)$-neighborhood of $Id$ in $G$.
}

{
Without loss of generality, we assume that $l\geq t$. For $s_0\in \mathbb{I}$, consider the interval
\begin{align*}
    \mathbb{I}(s_0)=(s_0-|\mathbb{I}|e^{-\frac{d(l+t)}{2}},s_0+|\mathbb{I}|e^{-\frac{d(l+t)}{2}})
\end{align*}
such that $\mathbb{I}(s_0)\subset \mathbb{I}$.
For any $s\in \mathbb{I}(s_0)$, and any $\mathbf{w}\in \{0\}\times \mathbb{I}^{r(d-r)-1}$,
\begin{align*}
    &|\psi_t(s \mathbf{E}_{11}+\mathbf{w})-\psi_t(s_0 \mathbf{E}_{11}+\mathbf{w})|
    \leq |\psi(a_t u_{\boldsymbol{\varphi}}(s\mathbf{E}_{11}+\mathbf{w})\Gamma)-\psi(a_t u_{\boldsymbol{\varphi}}(s_0 \mathbf{E}_{11}+\mathbf{w})\Gamma)|\\
    &+ |\psi(u(s\prm \mathbf{\mathbf{E}_{11}})a_t u_{\boldsymbol{\varphi}}(s\mathbf{\mathbf{E}_{11}}+\mathbf{w})\Gamma)-\psi(u(s\prm \mathbf{\mathbf{E}_{11}})a_t u_{\boldsymbol{\varphi}}(s_0\mathbf{\mathbf{E}_{11}}+\mathbf{w})\Gamma)|.
\end{align*}
Note that 
\begin{align*}
    &a_t u_{\boldsymbol{\varphi}}(s\mathbf{\mathbf{E}_{11}}+\mathbf{w})=a_t (u(s \mathbf{E}_{11}+\mathbf{w}),\boldsymbol{\varphi}(s \mathbf{E}_{11}+\mathbf{w}))\\
    &=a_t (u(s_0 \mathbf{E}_{11}+\mathbf{w}+(s-s_0)\mathbf{E}_{11}),\boldsymbol{\varphi}(s_0\mathbf{E}_{11}+\mathbf{w})+\boldsymbol{\varphi}(s \mathbf{E}_{11}+\mathbf{w})-\boldsymbol{\varphi}(s_0 \mathbf{E}_{11}+\mathbf{w}))\\
    &=(u(e^{dt}(s-s_0)\mathbf{\mathbf{E}_{11}}),e^{(d-r)t}(s-s_0)\partial_{11}\boldsymbol{\varphi})a_t u_{\boldsymbol{\varphi}}(s_0\mathbf{\mathbf{E}_{11}}+\mathbf{w}).
\end{align*}
where the last equality follows by mean value theorem (for simplicity of notations, by uniform boundedness of $||\partial_{11}\varphi||_{\infty}$ on $\mathcal{U}$,  we write $\partial_{11}\boldsymbol{\varphi}$ for $\partial_{11}\boldsymbol{\varphi}(\Tilde{s}\mathbf{\mathbf{E}_{11}}+\mathbf{w})$ with arbitrary $\Tilde{s}$ ).
}
{
Since
\begin{align*}
    e^{(d-r)t}|s-s_0|\leq e^{dt}|s-s_0|\leq e^{-\frac{d(l+t)}{2}}\cdot e^{dt}|\mathbb{I}|=e^{-\frac{d(l-t)}{2}}|\mathbb{I}|,
\end{align*}
we have
\begin{align*}
    a_t u_{\boldsymbol{\varphi}}(s\mathbf{\mathbf{E}_{11}}+\mathbf{w})&=O_G(e^{-\frac{d(l-t)}{2}}|\mathbb{I}|)a_t u_{\boldsymbol{\varphi}}(s_0\mathbf{\mathbf{E}_{11}}+\mathbf{w})\\
    &=O_G(e^{-\frac{d(l-t)}{2}})a_t u_{\boldsymbol{\varphi}}(s_0\mathbf{\mathbf{E}_{11}}+\mathbf{w}).
\end{align*}
Likewise,
\begin{align*}
     u(s\prm \mathbf{\mathbf{E}_{11}})a_t u_{\boldsymbol{\varphi}}(s\mathbf{\mathbf{E}_{11}}+\mathbf{w})=O_G(e^{-\frac{d(l-t)}{2}})u(s\prm \mathbf{\mathbf{E}_{11}})a_t u_{\boldsymbol{\varphi}}(s_0\mathbf{\mathbf{E}_{11}}+\mathbf{w}).
\end{align*}
Since $\psi\in C_c^{\infty}(G/\Gamma)$, $\psi$ is Lipschitz, and hence
\begin{align*}
    |\psi_t(s\mathbf{\mathbf{E}_{11}}+\mathbf{w})-\psi_t(s_0\mathbf{\mathbf{E}_{11}}+\mathbf{w})|=O(e^{-\frac{d(l-t)}{2}}).
\end{align*}
Therefore,
\begin{align}\label{align: estimate in Lemma 4.3}
    &\int_{\mathbb{I}(s_0)}\psi_t(s\mathbf{\mathbf{E}_{11}}+\mathbf{w})\psi_l(s\mathbf{\mathbf{E}_{11}}+\mathbf{w})ds\nonumber\\
    &=\int_{\mathbb{I}(s_0)}(\psi_t(s_0\mathbf{\mathbf{E}_{11}}+\mathbf{w})+O(e^{-\frac{d(l-t)}{2}}))\cdot\psi_l(s\mathbf{\mathbf{E}_{11}}+\mathbf{w})ds\nonumber\\
    &=\psi_t(s_0\mathbf{\mathbf{E}_{11}}+\mathbf{w})\int_{\mathbb{I}(s_0)}\psi_l(s\mathbf{\mathbf{E}_{11}}+\mathbf{w})ds+O(e^{-\frac{d(l-t)}{2}})\cdot|\mathbb{I}(s_0)|.
\end{align}
Now we estimate $\int_{\mathbb{I}(s_0)}\psi_l(s\mathbf{\mathbf{E}_{11}}+\mathbf{w})ds$. Note that 
\begin{align*}
    &u(s\prm \mathbf{E}_{11})a_l u_{\boldsymbol{\varphi}}(s\mathbf{\mathbf{E}_{11}}+\mathbf{w})\\
    &=a_l (u((s+e^{-dl}s\prm)\mathbf{\mathbf{E}_{11}}+\mathbf{w}),\boldsymbol{\varphi}(s\mathbf{\mathbf{E}_{11}}+\mathbf{w}))\\
    &=a_l (Id,\boldsymbol{\varphi}(s\mathbf{\mathbf{E}_{11}}+\mathbf{w})-\boldsymbol{\varphi}((e^{-dl}s\prm+s)\mathbf{\mathbf{E}_{11}}+\mathbf{w}))\cdot
    a_{-l}a_l u_{\boldsymbol{\varphi}}((s+e^{-dl}s\prm)\mathbf{\mathbf{E}_{11}}+\mathbf{w})\\
    &=a_l(Id,e^{-dl}s\prm \partial_{11} \boldsymbol{\varphi})a_{-l}a_l u_{\boldsymbol{\varphi}}((s+e^{-dl}s\prm)\mathbf{\mathbf{E}_{11}}+\mathbf{w})\\
    &=O_G(e^{-rl})a_l u_{\boldsymbol{\varphi}}((s+e^{-dl}s\prm)\mathbf{\mathbf{E}_{11}}+\mathbf{w}).
\end{align*}
As $\psi$ is Lipschitz,
\begin{align*}
    \psi_l(s\mathbf{\mathbf{E}_{11}}+\mathbf{w})=\psi(a_l u_{\boldsymbol{\varphi}}(s\mathbf{\mathbf{E}_{11}}+\mathbf{w})\Gamma)-\psi(a_l u_{\boldsymbol{\varphi}}((s+e^{-dl}s\prm)\mathbf{\mathbf{E}_{11}}+\mathbf{w})\Gamma)+O(e^{-rl}).
\end{align*}
Since $\mathbb{I}(s_0)$ and $\mathbb{I}(s_0)+e^{-dl}s\prm$ overlap except for a length of $O(e^{-dl})$, we have
\begin{align*}
    &\int_{\mathbb{I}(s_0)}\psi_l(s\mathbf{\mathbf{E}_{11}}+\mathbf{w})ds\\
    &=\int_{\mathbb{I}(s_0)}\psi(a_l u_{\boldsymbol{\varphi}}(s\mathbf{\mathbf{E}_{11}}+\mathbf{w})\Gamma)-\psi(a_l u_{\boldsymbol{\varphi}}((s+e^{-dl}s\prm)\mathbf{\mathbf{E}_{11}}+\mathbf{w})\Gamma)+O(e^{-rl})ds\\
    &=O(e^{-dl})+O(e^{-rl})|\mathbb{I}(s_0)|\\
    &=O(e^{-\frac{d(l-t)}{2}})|\mathbb{I}(s_0)|+O(e^{-(l-t)})|\mathbb{I}(s_0)|\\
    &=O(e^{-(l-t)})|\mathbb{I}(s_0)|.
\end{align*}
Now we consider the partition $\mathbb{I}=\bigcup_{j=1}^p \mathbb{I}_j$ such that $\mathbb{I}_j=[s_{j-1},s_j]$ with $s_j-s_{j-1}=2e^{-\frac{d(l+t)}{2}}|\mathbb{I}|$ for $1\leq j\leq p-1$, and $s_p-s_{p-1}\leq 2e^{-\frac{d(l+t)}{2}}|\mathbb{I}|$. By (\ref{align: estimate in Lemma 4.3}), we have
\begin{align*}
    &\int_{\mathbb{I}} \psi_t(s\mathbf{\mathbf{E}_{11}}+\mathbf{w})\psi_l(s\mathbf{\mathbf{E}_{11}}+\mathbf{w})ds\\
    &=\sum_{j=1}^p\int_{\mathbb{I}_j} \psi_t(s\mathbf{\mathbf{E}_{11}}+\mathbf{w})\psi_l(s\mathbf{\mathbf{e}_{11}}+\mathbf{w})ds\\
    &=\sum_{j=1}^{p-1}\int_{\mathbb{I}_j}\psi_t(s\mathbf{\mathbf{E}_{11}}+\mathbf{w})\psi_l(s\mathbf{\mathbf{E}_{11}}+\mathbf{w})ds+O(e^{-\frac{d(l+t)}{2}})|\mathbb{I}|\\
    &=\sum_{j=1}^{p-1}O(e^{-(l-t)})|\mathbb{I}_j|+O(e^{-\frac{d(l+t)}{2}})|\mathbb{I}|\\
    &=O(e^{-(l-t)})|\mathbb{I}|=O(e^{-(l-t)}).
\end{align*}
The above estimate holds for any $\mathbf{w}\in \{\mathbf{0}\}\times \mathbb{I}^{r(d-r)-1}$. Now the lemma follows from the above estimate and Fubini's theorem.}
\end{proof}

\section{Margulis' height function}\label{section Margulis' height function}
In this section, we will recall the definition of Margulis' height function on $SL_d(\mathbb{R})/SL_d(\mathbb{Z})$ and its uniform contraction property . 

Margulis' height function was first introduced in \cite{Eskin_upperboundsandasymptoticsinaquantitativeversion_1998} and later developed in several papers (see for example \cite{Benoist_Quint_Random_Walks_on_Finite_volume_2012}\cite{Shi_Pointwiseequidistribution_2020}). It measures the depth of elements of $X$ into the cusps.
It has been used to study equidistribution problem for certain unbounded functions (cf.\cite{Eskin_upperboundsandasymptoticsinaquantitativeversion_1998}\cite{Eskin_quadraticformsofsignature_2005}\cite{Magulis_quantitativeversionoftheoppenheim_2011}) and random walks on homogeneous spaces (cf.\cite{Benoist_Quint_Random_Walks_on_Finite_volume_2012}\cite{Eskin_Margulis_Recurrence_properties_of_random_walks_2004}).

We start with the vector space $\mathbf{V}=\wedge^*\bR^d=\bigoplus_{0\leq i\leq d}\wedge^i \bR^d$, where $G\prm=SL_d(\bR)$ acts on $\mathbf{V}$ naturally.

Let $\Delta$ be a lattice in $\bR^d$. We say that a subspace $L$ of $\bR^d$ is $\Delta$-rational if $L\cap \Delta$ is a lattice in $L$. For any $\Delta$-rational subspace $L$, denote $d(L)$ or $d_{\Delta}(L)$ the volume of $L/L\cap \Delta$.
Note that $d(L)$ is the norm of $u_1\wedge u_2\wedge\cdots\wedge u_l$ in $\mathbf{V}$, where $\{u_i\}_{1\leq i \leq l}$ is a $\bZ$-basis of $L\cap \Delta$. If $L=\{\mathbf{0}\}$, we set $d(L)=1$.

For any lattice $\Delta$, we define for $0\leq i\leq d$,
\begin{align*}
    \alpha_i(\Delta):=\sup\left\{\frac{1}{d(L)}: L \text{ is a } \Delta\text{-rational subspace of dimension }i \right \}.
\end{align*}

\begin{proposition}\label{contraction hypothesis for Margulis' height function}
There exists a continuous map $\Tilde{\alpha}:SL_d(\bR)/SL_d(\mathbb{Z})\to [1,\infty]$ and $b_1>0$ such that for $B$ a bounded open box in $Mat_{r\times(d-r)}(\bR)$, for all $t>0$ large enough and for any unimodular lattice $\Lambda$ of $\bR^d$,
\begin{align*}
    \frac{1}{|B|}\int_{B}\Tilde{\alpha}(a_t u(\mathbf{s})\Lambda)d\mathbf{s} < 2^{-r(d-r)-2}\Tilde{\alpha}(\Lambda)+b_1,
    \end{align*}
and there exists $\nu>0$ such that
\begin{align*}
 \alpha_1(\Lambda)^{\nu}\leq \Tilde{\alpha}(\Lambda).
\end{align*}
Moreover, a measurable subset $K$ of $SL_d(\mathbb{R})/SL_d(\mathbb{Z})$ is precompact if there exists $N>0$ such that
\begin{align*}
    K\subset \{x\in X: \Tilde{\alpha}(x)\leq N\}.
\end{align*}
\end{proposition}

\begin{proof}
Define $\Tilde{\alpha}=\epsilon^{-q(1)}\cdot \sum_{i=0}^d \epsilon^{q(i)}\cdot \alpha_i^{\nu}$, where $\epsilon,\nu>0$ are sufficiently small numbers and $q(i)=i(d-i)$. The fact that $\Tilde{\alpha}$ satisfies conclusion of Proposition \ref{contraction hypothesis for Margulis' height function} will follow from \cite[Lemma 4.1]{Shi_Pointwiseequidistribution_2020}.
\end{proof}
The function $\Tilde{\alpha}$ above is the Margulis' height function that we need in our setting.

\begin{remark}\label{remark: Lipschitz property of Margulis height function}
The function $\Tilde{\alpha}$ satisfies Lipschitz property as follows: For any bounded neighborhood $\mathcal{V}$ of $e$ of $SL_d(\bR)$, there exists $\overline{M}>0$ such that for any $x\in SL_d(\bR)/SL_d(\bZ)$, any $g\in \mathcal{V}$,
\begin{align*}
    \Tilde{\alpha}(gx)\leq \overline{M} \Tilde{\alpha}(x).
\end{align*}
Indeed, $\overline{M}>0$ is the maximum of operator norms of elements in $\mathcal{V}$ acting on $\mathbf{V}$.
\end{remark}

\section{Mixed height function}\label{section mixed height functions}

In this section, we will construct a mixed height function, which is crucial for us to prove Proposition \ref{limit measure spend very little measure on singular sets}. The main result of this section is the following:
\begin{proposition}\label{contraction of mixed height functions}
Let $\boldsymbol{\varphi}$ be a $C^1$ map from  $\mathcal{U}$ to $(\bR^d)^k$ satisfying $(\boldsymbol{\varphi}(\mathbf{s}))_{>r}\equiv \mathbf{0}$ for any $\mathbf{s}\in \mathcal{U}$. For any $\mathbf{m}\in \bZ^k\setminus \{\mathbf{0}\}$, any closed cube $I\subset \mathcal{U}\setminus Bad_{\mathbf{m}}$ (for $Bad_{\mathbf{m}}$, see (\ref{genericity condition under regularity})), there are $t>0$ sufficiently large (depending on $I$, $\mathbf{m}$) and measurable function $\beta_{\mathbf{m}}:G/\Gamma\to (0,\infty]$ such that the following hold:

(1) For any $l>0$, $\{x\in G/\Gamma:\beta_{\mathbf{m}}(x)\leq l\}$ is compact;

(2) For any $x\in G/\Gamma$, $\beta_{\mathbf{m}}(x)=\infty$ if and only if $x\in X_{\mathbf{m}}$;

(3) Given any $n\in \bZ_{\geq 0}$, a box $J\subset I$ with $J=\prod_{i=1}^{r(d-r)}J_i$, where $J_i\subset \bR$ and $|J_i|\leq 2 e^{-dnt}$ for all $i$. There exists $\Tilde{M}_1>0$ such that for any $\mathbf{s},\Tilde{\mathbf{s}}\in J$, one has
\begin{align*}
    \beta_{\mathbf{m}}(a_{nt}u_{\boldsymbol{\varphi}}(\Tilde{\mathbf{s}})\Gamma)\leq \Tilde{M}_1 \beta_{\mathbf{m}}(a_{nt}u_{\boldsymbol{\varphi}}(\mathbf{s})\Gamma);
\end{align*}

(4) There exists $\Tilde{M}_2>0$ {depending on $t$}, for any $n\in \bZ_{\geq 0}$, any $\mathbf{s}\in I$ and any $\tau\in \bR$ with $|\tau|\leq t$, one has
\begin{align*}
    \beta_{\mathbf{m}}(a_{\tau}a_{nt}u_{\boldsymbol{\varphi}}(\mathbf{s})\Gamma)\leq \Tilde{M}_2\beta_{\mathbf{m}}(a_{nt}u_{\boldsymbol{\varphi}}(\mathbf{s})\Gamma);
\end{align*}

(5) There exists $b>0$ such that the following holds: for any $n\in \bZ_{\geq 0}$ and any box $J\subset I$ with $J=\prod_{i=1}^{r(d-r)}J_i$ satisfying either $n\geq 1$ and $|J_i|\geq e^{-dnt}$ for all $1\leq i \leq r(d-r)$, or $n=0$ and $J=I$, one has
\begin{align*}
    \int_J \beta_{\mathbf{m}}(a_{(n+1)t} u_{\boldsymbol{\varphi}}(\mathbf{s})\Gamma)d\mathbf{s}\leq \frac{1}{2}\int_{J}\beta_{\mathbf{m}}(a_{nt}u_{\boldsymbol{\varphi}}(\mathbf{s})\Gamma)d\mathbf{s}+b|J|.
\end{align*}
\end{proposition}
\begin{remark}
The function $\beta_{\mathbf{m}}$ in Proposition \ref{contraction of mixed height functions} is the desired mixed height function.
\end{remark}
{From now on until the end of this section, we will fix a closed cube $I\subset \mathcal{U}\setminus Bad_{\mathbf{m}}$}. Recall that the {finite} number $M_1$ is defined as in (\ref{definition of M1}). By the choice of $I$, and
the fact that there are only finitely many $\mathbf{a}\in \bZ^{d-r}$ satisfying $\norm{\mathbf{a}}_{\infty}\leq M_1 \norm{\mathbf{m}}$, we obtain {$\sigma>0$ such that} 
\begin{align}\label{proof of consequence of regularity condition}
    \inf_{\mathbf{s}\in I}\{\norm{(\boldsymbol{\varphi}(\mathbf{s}))_{\leq r}\cdot \mathbf{m}-\mathbf{s}
    \cdot \mathbf{a}-\mathbf{b}}_{\infty}
    : \norm{\mathbf{a}}_{\infty}\leq M_1 \norm{\mathbf{m}}, \mathbf{a}\in \bZ^{d-r},\mathbf{b}\in \bZ^r
    \}=\sigma.
\end{align}

\begin{remark}\label{remark: a small neighborhood of I with inf>0}
By (\ref{proof of consequence of regularity condition}), we can choose a closed neighborhood $I\prm$ of $I$ such that $I\prm$ is a closed cube contained in $\mathcal{U}$ and satisfies
\begin{align*}
    \inf_{\mathbf{s}\in I\prm}\{\norm{(\boldsymbol{\varphi}(\mathbf{s}))_{\leq r} \cdot \mathbf{m}-\mathbf{s}
    \cdot \mathbf{a}-\mathbf{b}}_{\infty}
    : \norm{\mathbf{a}}_{\infty}\leq M_1 \norm{\mathbf{m}}, \mathbf{a}\in \bZ^{d-r},\mathbf{b}\in \bZ^r
    \}=\frac{\sigma}{2}.
\end{align*}
\end{remark}

Next we construct a suitable function measuring the distance to the closed submanifold $X_{\mathbf{m}}$.
For $\mathbf{m}\in \bZ^k\setminus \{\mathbf{0}\}$, consider the quotient space
\begin{align*}
    (\bR^d)^k_{\mathbf{m}}:=\{\mathbf{v}\in (\bR^d)^k:\mathbf{v}\cdot \mathbf{m}\in \bZ^d\}/\sim,
\end{align*}
where $\mathbf{v}\sim \mathbf{v}\prm$ if and only if $\mathbf{v}\cdot \mathbf{m}=\mathbf{v}\prm \cdot \mathbf{m}$. One can directly verify that $\sim$ is an equivalence relation.

\begin{lemma}\label{at most one vector in the equivalence space}
For any $(g,\mathbf{v})\in G$, there exists at most one $\mathbf{v}_0\in (\bR^d)^k_{\mathbf{m}}$ such that
\begin{align}\label{controlled by margulis height function}
    \norm{(\mathbf{v}-g\mathbf{v}_0)\cdot \mathbf{m}}<\frac{1}{2}\inf_{\mathbf{w}\in \bZ^d\setminus \{\mathbf{0}\}}\norm{g\mathbf{w}}.
\end{align}
\end{lemma}

\begin{proof}
Suppose there are two vectors $\mathbf{v}_0$ and $\mathbf{v}_0\prm$ in $(\bR^d)^k_{\mathbf{m}}$ satisfying (\ref{controlled by margulis height function}) such that $\mathbf{v}_0\not\sim \mathbf{v}_0\prm$. Then 
\begin{align*}
    \norm{g(\mathbf{v}_0-\mathbf{v}_0\prm)\cdot \mathbf{m}}\leq \norm{(\mathbf{v}-g\mathbf{v}_0)\cdot\mathbf{m}}+\norm{(\mathbf{v}-g\mathbf{v}_0\prm)\cdot\mathbf{m}}
    < \inf_{\mathbf{w}\in \bZ^d\setminus\{\mathbf{0}\}}\norm{g\mathbf{w}}.
\end{align*}
But since $\mathbf{v}_0\not\sim \mathbf{v}_0\prm$, $(\mathbf{v}_0-\mathbf{v}_0\prm)\cdot \mathbf{m}\in \bZ^d\setminus\{\mathbf{0}\}$. This is a contradiction.
\end{proof}

\begin{definition}\label{definition of alpha_m}
Let $\mathbf{m}\in \bZ^k\setminus\{\mathbf{0}\}$. For any $(g,\mathbf{v})\in G$, we say that $\boldsymbol{\xi}_{g,v}\in (\bR^d)_{\mathbf{m}}^k$ exists if $\boldsymbol{\xi}_{g,v}$ satisfies (\ref{controlled by margulis height function}) in the place of $\mathbf{v}_0$. By convention, we set $\boldsymbol{\xi}_{g,\mathbf{v}}=\infty$ if it does not exist. Define the function
\begin{align}\label{definition of alpha m}
    \alpha_{\mathbf{m}}(g,\mathbf{v})=
    \begin{cases}
    \norm{(\mathbf{v}-g\boldsymbol{\xi}_{g,\mathbf{v}})\cdot \mathbf{m}}^{-1}, & \text{ If }\boldsymbol{\xi}_{g,\mathbf{v}} \text{ exists; }\\
    1, &\text{ if }  \boldsymbol{\xi}_{g,\mathbf{v}}=\infty.
    \end{cases}
\end{align}
\end{definition}

\begin{remark}\label{remark:consequence of existence of xi}
By Lemma \ref{at most one vector in the equivalence space}, $\alpha_{\mathbf{m}}$ is a well-defined function on $G$. {By Minkowski's first theorem, there is a constant $0<\mu_d\leq 1$ such that if $\boldsymbol{\xi}_{g,\mathbf{v}}$ exists for $(g,\mathbf{v})$, then $\alpha_{\mathbf{m}}(g,\mathbf{v})>\mu_d$. Therefore, by definition of $\alpha_{\mathbf{m}}$, we have $\alpha_{\mathbf{m}}(g,\mathbf{v})>\mu_d$ for any $(g,\mathbf{v})\in G$.} Moreover, by (\ref{proof of consequence of regularity condition}), for any $\mathbf{s}\in I$, $\alpha_{\mathbf{m}}(u_{\boldsymbol{\varphi}}(\mathbf{s}))\leq \sigma^{-1}$.

\end{remark}
\begin{lemma}\label{alpha is well defined and lowe semicontinuous}
$\alpha_{\mathbf{m}}$ is a well-defined function on $G/\Gamma$. Moreover, $\alpha_{\mathbf{m}}$ is lower semi-continuous.
\end{lemma}

\begin{proof}
Take any $(g,\mathbf{v})\in G$ and any $(\gamma,\mathbf{v}\prm)\in \Gamma$.

If $\boldsymbol{\xi}_{g,\mathbf{v}}$ exists, then $\gamma^{-1}(\boldsymbol{\xi}_{g,\mathbf{v}}+\mathbf{v}\prm)\in (\bR^d)^k_{\mathbf{m}}$. Note that 
\begin{align*}
    \norm{(\mathbf{v}+g\mathbf{v}\prm-g\gamma\cdot \gamma^{-1}(\boldsymbol{\xi}_{g,\mathbf{v}}+\mathbf{v}\prm))\cdot \mathbf{m}}=\norm{(\mathbf{v}-g\boldsymbol{\xi}_{g,\mathbf{v}})\cdot \mathbf{m}}
    <\frac{1}{2} \inf_{\mathbf{w}\in \bZ^d\setminus\{\mathbf{0}\}}\norm{g\mathbf{w}}.
\end{align*}
Thus $\boldsymbol{\xi}_{g\gamma,\mathbf{v}+g\mathbf{v}\prm}=\gamma(\boldsymbol{\xi}_{g,\mathbf{v}}+\mathbf{v}\prm)$ exists and $\alpha_{\mathbf{m}}(g,\mathbf{v})=\alpha_{\mathbf{m}}(g\gamma,\mathbf{v}+g\mathbf{v}\prm)$. 

If $\boldsymbol{\xi}_{g,\mathbf{v}}$ does not exist, same argument as above shows that $\boldsymbol{\xi}_{g\gamma,\mathbf{v}+g\mathbf{v}\prm}$ does not exist neither. 

If $\boldsymbol{\xi}_{g,\mathbf{v}}$ exists, it is locally constant. Therefore, $\alpha_{\mathbf{m}}$ is lower semi-continuous.
\end{proof}

Let $\nu\in (0,\frac{1}{r(d-r)})$ be a number satisfying Proposition \ref{contraction hypothesis for Margulis' height function}. Let $c=4\cdot (10r^2d)^{\nu}\cdot 2^{r(d-r)}$ and $t>0$ be a sufficiently large number (to be specified later). 
Define
\begin{align}\label{definition of beta m}
    \beta_{\mathbf{m}}=\alpha_{\mathbf{m}}^{\nu}+ce^{\nu rt}\Tilde{\alpha}.
\end{align}
We will prove that $\beta_{\mathbf{m}}$ satisfies properties (1)-(5) of Proposition \ref{contraction of mixed height functions}.

\begin{proof}[Proof of Proposition \ref{contraction of mixed height functions} (1)]

Since
\begin{align*}
    \{x\in X:\beta_{\mathbf{m}}(x)\leq l\}\subset \{x\in X:\Tilde{\alpha}(x)\leq lc^{-1}e^{-\nu rt}\}, 
\end{align*}
$\{x\in X:\beta_{\mathbf{m}}(x)\leq l\}$ is precompact.
As $\Tilde{\alpha}$ is continuous and $\alpha_{\mathbf{m}}$ is lower semicontinuous, $\{x\in X:\beta_{\mathbf{m}}(x)\leq l\}$ is closed and thus compact.
\end{proof}

\begin{proof}[Proof of Proposition \ref{contraction of mixed height functions} (2)]

If for some $x\in G/\Gamma$, $\beta_{\mathbf{m}}(x)=\infty$, then $\alpha_{\mathbf{m}}(x)=\infty$ as $\Tilde{\alpha}(x)<\infty$.
Let $x=(g,\mathbf{v})\Gamma/\Gamma$. By definition of $\alpha_{\mathbf{m}}$, we have $(\mathbf{v}-g\boldsymbol{\xi})\cdot \mathbf{m}=\mathbf{0}$ for some $\boldsymbol{\xi}\in (\bR^d)^k_{\mathbf{m}}$. Note that $g^{-1}\mathbf{v}\cdot \mathbf{m}=\boldsymbol{\xi}\cdot \mathbf{m}\in \bZ^d$. Hence $(g,\mathbf{v})\Gamma=(g,g g^{-1}\mathbf{v})\Gamma\in X_{\mathbf{m}}$.

{Conversely, if $x=(g,\mathbf{v})\Gamma/\Gamma\in X_{\mathbf{m}}$, then by definition of $X_{\mathbf{m}}$, we have $g^{-1}\mathbf{v}\cdot \mathbf{m}\in \mathbb{Z}^d$. Choose any $\boldsymbol{\xi}\in (\mathbb{R}^d)^k$ such that $g^{-1}\mathbf{v}\cdot \mathbf{m}=\boldsymbol{\xi}\cdot \mathbf{m}$, then $(\mathbf{v}-g \boldsymbol{\xi})\cdot \mathbf{m}=g(g^{-1}\mathbf{v}\cdot \mathbf{m}-\boldsymbol{\xi}\cdot \mathbf{m})=\mathbf{0}$. Therefore, $\alpha_{\mathbf{m}}(x)=\infty.$}
\end{proof}

\textbf{Notations.} Let's fix some simplified notations for the rest of the proof. We will fix an $\mathbf{m}\in \bZ^k\setminus \{\mathbf{0}\}$ till the end of this section. In the following, $t>0$ is a sufficiently large number.
\begin{itemize}
    \item For any $n\in \mathbb{N}$, any $\mathbf{s}\in \mathcal{U}$, denote $a_{nt}u_{\boldsymbol{\varphi}}(\mathbf{s})=(g_n(\mathbf{s}),\mathbf{v}_n(\mathbf{s}))$.
    
    \item If $\boldsymbol{\xi}_{g_n(\mathbf{s}),\mathbf{v}_n(\mathbf{s})} $ exists, denote $\boldsymbol{\xi}_{g_n(\mathbf{s}),\mathbf{v}_n(\mathbf{s})}=\boldsymbol{\xi}_{n,\mathbf{s}}$.
    
    \item For any $\mathbf{v}\in (\bR^d)^k_{\mathbf{m}}$, $n\in \mathbb{N}$, $\mathbf{s}\in \mathcal{U}$, let 
\begin{align*}
    w(n,\mathbf{s},\mathbf{v})&=(\mathbf{v}_n(\mathbf{s})-g_n(\mathbf{s})\mathbf{v})\cdot \mathbf{m}\\
    &=\begin{bmatrix}
        e^{(d-r)nt}[(\boldsymbol{\varphi}(\mathbf{s}))_{\leq r}-(\mathbf{v})_{\leq r}-\mathbf{s}\cdot (\mathbf{v})_{>r}]\cdot \mathbf{m}\\
        e^{-rnt}(\mathbf{v})_{>r}\cdot \mathbf{m}
    \end{bmatrix}
    =\begin{bmatrix}
        w_1(n,\mathbf{s},\mathbf{v})\\
        \vdots\\
        w_d(n,\mathbf{s},\mathbf{v})
    \end{bmatrix}.
    \end{align*}
We note that if $\boldsymbol{\xi}_{n,\mathbf{s}}$ exists, then $\alpha_{\mathbf{m}}(a_{nt}u_{\boldsymbol{\varphi}}(\mathbf{s}))=\norm{w(n,\mathbf{s},\boldsymbol{\xi}_{n,\mathbf{s}})}^{-1}$.

\item For any differentiable function $\psi:Mat_{r\times(d-r)}(\bR)\to \bR$, by mean value theorem in several variables, for any $\mathbf{s},\Tilde{\mathbf{s}}\in Mat_{r\times (d-r)}(\bR)$, there is a $\hat{\mathbf{s}}$ such that {
    \begin{align*}
        \psi(\Tilde{\mathbf{s}})-\psi(\mathbf{s})=\sum_{i=1}^r\sum_{j=1}^{d-r}\frac{\partial \psi}{\partial s_{ij}}(\hat{\mathbf{s}})\cdot(\Tilde{s}_{ij}-s_{ij}).
    \end{align*}}
{Since the functions that we consider have bounded first derivative on a bounded set,} we will omit this $\hat{\mathbf{s}}$ for simplicity.

\item For $1\leq i \leq d-r$, let $\mathbf{e}_i$ denote the column vector in $\bR^{d-r}$ with $1$ in $i$-th row and $0$ elsewhere. {Let $<,>$ denote the usual inner product of column vectors}.
\end{itemize}

\begin{lemma}\label{consequence of existence of xi}
Let $n\geq 1$ be an integer and {$t>log(2\mu_d^{-1} \sigma^{-1})$, where $\mu_d>0$ is the constant given as in Remark \ref{remark:consequence of existence of xi}.} For any $\mathbf{s}\in I\prm$, where $I\prm$ is given as in Remark \ref{remark: a small neighborhood of I with inf>0}, if $\boldsymbol{\xi}_{n,\mathbf{s}}$ exists, then $\norm{(\boldsymbol{\xi}_{n,\mathbf{s}})_{>r}\cdot \mathbf{m}}_{\infty}>M_1\norm{\mathbf{m}}$. 
\end{lemma}

\begin{proof}
Suppose $\norm{(\boldsymbol{\xi}_{n,\mathbf{s}})_{>r}\cdot \mathbf{m}}_{\infty}\leq M_1\norm{\mathbf{m}}$. By definition of $I\prm$,
\begin{align*}
    \norm{(\mathbf{v}_n(\mathbf{s})-g_n(\mathbf{s})\boldsymbol{\xi}_{n,\mathbf{s}})\cdot \mathbf{m}}&\geq \norm{(\mathbf{v}_n(\mathbf{s})-g_n(\mathbf{s})\boldsymbol{\xi}_{n,\mathbf{s}})_{\leq r}\cdot \mathbf{m}}_{\infty}\\
    &\geq e^{(d-r)nt}\norm{[(\boldsymbol{\varphi}(\mathbf{s}))_{\leq r}-(\boldsymbol{\xi}_{n,\mathbf{s}})_{\leq r}-\mathbf{s} (\boldsymbol{\xi}_{n,\mathbf{s}})_{>r}]\cdot \mathbf{m}}_{\infty}\\
    &\geq e^t \frac{\sigma}{2}>\mu_d^{-1}.
\end{align*}
By Remark \ref{remark:consequence of existence of xi}, this contradicts the existence of $\boldsymbol{\xi}_{n,\mathbf{s}}$.
\end{proof}

\begin{proof}[Proof of Proposition \ref{contraction of mixed height functions} (3)]

If $n=0$, by Remark \ref{remark:consequence of existence of xi}, we have ${\mu_d^{\nu}}\leq \alpha_{\mathbf{m}}^{\nu}(u_{\boldsymbol{\varphi}}(\Tilde{\mathbf{s}}))\leq \sigma^{-\nu}$. Since $\Tilde{\alpha}$ is continuous and bounded on compact sets, there exists $0<m<M$ such that for any $\mathbf{s}\in I$,
\begin{align*}
    m\leq \Tilde{\alpha}(u_{\boldsymbol{\varphi}}(\mathbf{s})\Gamma)\leq M.
\end{align*}
Therefore, 
\begin{align*}
    \beta_{\mathbf{m}}(u_{\boldsymbol{\varphi}}(\Tilde{\mathbf{s}})\Gamma)\leq \sigma^{-\nu}+ce^{\nu rt}\cdot M\leq { \frac{\sigma^{-\nu}+ce^{\nu rt}M}{\mu_d^{\nu}+ce^{\nu rt}m}\beta_{\mathbf{m}}(u_{\boldsymbol{\varphi}}(\mathbf{s})\Gamma).}
\end{align*}
Assume $n\geq 1$. If $\boldsymbol{\xi}_{n,\Tilde{\mathbf{s}}}$ does not exist, then by definition, $\alpha_{\mathbf{m}}^{\nu}(a_{nt}u_{\boldsymbol{\varphi}}(\Tilde{\mathbf{s}})\Gamma)=1$. Now we suppose that $\boldsymbol{\xi}_{n,\Tilde{\mathbf{s}}}$ exists.

\textbf{Case 1}: $\norm{(w(n,\mathbf{s},\boldsymbol{\xi}_{n,\Tilde{\mathbf{s}}}))_{\leq r}}_{\infty}\geq N_2\norm{(w(n,\mathbf{s},\boldsymbol{\xi}_{n,\Tilde{\mathbf{s}}}))_{>r}}_{\infty}$, where $N_2=2(r+1)(d-r)$.

Choose $t>log(2\mu_d^{-1} \sigma^{-1})$, by Lemma \ref{consequence of existence of xi}, the Lipschity continuity of $\boldsymbol{\varphi}$, and the choices of $M_1$ and $J$, we have
\begin{align*}
    \norm{(w(n,\Tilde{\mathbf{s}},\boldsymbol{\xi}_{n,\Tilde{\mathbf{s}}}))_{\leq r}-(w(n,\mathbf{s},\boldsymbol{\xi}_{n,\Tilde{\mathbf{s}}}))_{\leq r}}_{\infty}&\leq 2e^{-rnt}(r+1)(d-r)\norm{(\boldsymbol{\xi}_{n,\Tilde{\mathbf{s}}})_{>r}\cdot \mathbf{m}}_{\infty}.\\
    &=2(r+1)(d-r)\norm{(w(n,\Tilde{\mathbf{s}},\boldsymbol{\xi}_{n,\Tilde{\mathbf{s}}}))_{> r}}_{\infty}.
\end{align*}
Hence, by the assumption of Case 1,
\begin{align*}
    \norm{w(n,\Tilde{\mathbf{s}},\boldsymbol{\xi}_{n,\Tilde{\mathbf{s}}})}&\geq \norm{(w(n,\Tilde{\mathbf{s}},\boldsymbol{\xi}_{n,\Tilde{\mathbf{s}}}))_{\leq r}}_{\infty} \\
    &\geq \norm{(w(n,\mathbf{s},\boldsymbol{\xi}_{n,\tilde{\mathbf{s}}}))_{\leq r}}_{\infty}-2(r+1)(d-r)\norm{(w(n,\mathbf{s},\boldsymbol{\xi}_{n,\tilde{\mathbf{s}}}))_{> r}}_{\infty}\\
    &\geq { \frac{1}{d^2} \norm{w(n,\mathbf{s},\boldsymbol{\xi}_{n,\tilde{\mathbf{s}}})}.}
\end{align*}

\textbf{Case 2}: $\norm{(w(n,\mathbf{s},\boldsymbol{\xi}_{n,\Tilde{\mathbf{s}}}))_{\leq r}}_{\infty}< N_2\norm{(w(n,\mathbf{s},\boldsymbol{\xi}_{n,\Tilde{\mathbf{s}}}))_{>r}}_{\infty}$. Then
\begin{align*}
    \norm{w(n,\Tilde{\mathbf{s}},\boldsymbol{\xi}_{n,\Tilde{\mathbf{s}}})}&\geq \norm{(w(n,\tilde{\mathbf{s}},\boldsymbol{\xi}_{n,\Tilde{\mathbf{s}}}))_{>r}}_{\infty}=\norm{(w(n,\mathbf{s},\boldsymbol{\xi}_{n,\Tilde{\mathbf{s}}}))_{>r}}_{\infty}\\
    &\geq \frac{1}{\sqrt{rN_2^2+d-r}}\norm{w(n,\mathbf{s},\boldsymbol{\xi}_{n,\Tilde{\mathbf{s}}})}\geq \frac{1}{5dr^2}\norm{w(n,\mathbf{s},\boldsymbol{\xi}_{n,\Tilde{\mathbf{s}}})}.
\end{align*}
{By construction of $\boldsymbol{\xi}_{n,\mathbf{s}}$, we have $\norm{w(n,\mathbf{s},\boldsymbol{\xi}_{n,\Tilde{\mathbf{s}}})}\geq \norm{w(n,\mathbf{s},\boldsymbol{\xi}_{n,\mathbf{s}})}$. Combining Case 1 and Case 2, we have 
\begin{align*}
    \norm{w(n,\Tilde{\mathbf{s}},\boldsymbol{\xi}_{n,\Tilde{\mathbf{s}}})}\geq \min\{\frac{1}{5d r^2},\frac{1}{d^2}\} \norm{w(n,\mathbf{s},\boldsymbol{\xi}_{n,\mathbf{s}})}.
\end{align*}
Therefore,
}
\begin{align*}
    \alpha_{\mathbf{m}}^{\nu}(a_{nt}u_{\boldsymbol{\varphi}}(\Tilde{\mathbf{s}})\Gamma)\leq \max(d^{2\nu},(5dr^2)^{\nu})\cdot \alpha_{\mathbf{m}}^{\nu}(a_{nt}u_{\boldsymbol{\varphi}}(\mathbf{s})\Gamma).
\end{align*}
For $\Tilde{\alpha}$, by Remark \ref{remark: Lipschitz property of Margulis height function}, we have for large enough $\overline{M}>0$,
\begin{align*}
    \Tilde{\alpha}(a_{nt}u_{\boldsymbol{\varphi}}(\Tilde{\mathbf{s}})\Gamma)
    =\Tilde{\alpha}(u(e^{dnt}(\Tilde{\mathbf{s}}-\mathbf{s}))a_{nt}u_{\boldsymbol{\varphi}}(\mathbf{s})\Gamma)
   \leq \overline{M}\Tilde{\alpha}(a_{nt}u_{\boldsymbol{\varphi}}(\mathbf{s})\Gamma).
    \end{align*}
By the above, let { $\Tilde{M}_1=2\max\{\overline{M},d^{2\nu} ,(5dr^2)^{\nu} ,\frac{\sigma^{-\nu}+ce^{\nu rt}M}{\mu_d^{\nu}+ce^{\nu rt}m}\}$}, then 
\begin{align*}
    \beta_{\mathbf{m}}(a_{nt}u_{\boldsymbol{\varphi}}(\Tilde{\mathbf{s}})\Gamma)\leq \Tilde{M}_1\beta_{\mathbf{m}}(a_{nt}u_{\boldsymbol{\varphi}}(\mathbf{s})\Gamma), \forall n\in \bZ_{\geq 0}.
\end{align*}
\end{proof}

\begin{proof}[Proof of Proposition \ref{contraction of mixed height functions} (4)]

By Remark \ref{remark: Lipschitz property of Margulis height function}, since $|\tau|\leq t$, we have 
\begin{align*}
    \Tilde{\alpha}(a_{\tau}a_{nt}u_{\boldsymbol{\varphi}}(\mathbf{s})\Gamma)\leq e^{r(d-r)\nu t}\Tilde{\alpha}(a_{nt}u_{\boldsymbol{\varphi}}(\mathbf{s})\Gamma).
\end{align*}
Let $a_{\tau}a_{nt}u_{\boldsymbol{\varphi}}(\mathbf{s})=(a_{\tau}g_n(\mathbf{s}),a_{\tau}\mathbf{v}_n(\mathbf{s}))$.

If $\boldsymbol{\xi}_{a_{\tau}g_n(\mathbf{s}),a_{\tau}\mathbf{v}_n(\mathbf{s})}$ does not exist, then $\alpha_{\mathbf{m}}(a_{\tau}a_{nt}u_{\boldsymbol{\varphi}}(\mathbf{s})\Gamma)=1$.

If $\boldsymbol{\xi}_{a_{\tau}g_n(\mathbf{s}),a_{\tau}\mathbf{v}_n(\mathbf{s})}=\mathbf{v}$ exists, then 
\begin{align*}
    \alpha_{\mathbf{m}}^{\nu}(a_{\tau}a_{nt}u_{\boldsymbol{\varphi}}(\mathbf{s})\Gamma)&=\norm{(a_{\tau}(\mathbf{v}_n(\mathbf{s})-g_n(\mathbf{s})\mathbf{v})\cdot \mathbf{m}}^{-\nu}\\
    &\leq e^{(d-r)\nu t}\norm{(\mathbf{v}_n(\mathbf{s})-g_n(\mathbf{s})\mathbf{v})\cdot \mathbf{m}}^{-\nu}\\
    &\leq\begin{cases}
    e^{(d-r)\nu t}\alpha_{\mathbf{m}}^{\nu}(a_{nt}u_{\boldsymbol{\varphi}}(\mathbf{s})\Gamma) &\text{If }\mathbf{v}=\boldsymbol{\xi}_{g_n(\mathbf{s}),\mathbf{v}_n(\mathbf{s})}\\
    e^{(d-r)\nu t}2^{\nu}\Tilde{\alpha}(a_{nt}u_{\boldsymbol{\varphi}}(\mathbf{s})\Gamma) &\text{If }\mathbf{v}\neq \boldsymbol{\xi}_{g_n(\mathbf{s}),\mathbf{v}_n(\mathbf{s})}
    \end{cases}
\end{align*}
Let {
\begin{align*}
    \Tilde{M}_2=\max(e^{(d-r)\nu t},\frac{2^{\nu}e^{(d-r)\nu t}+c\cdot e^{r(d-r)\nu t} e^{\nu rt}}{c\cdot e^{\nu r t}}),
\end{align*}}
then by the above estimates, Proposition \ref{contraction of mixed height functions} (4) is proved.
\end{proof}

To prove property (5) of Proposition \ref{contraction of mixed height functions}, we record the following lemmas:

\begin{lemma}\cite[Lemma 4.8]{Shi_Pointwiseequidistribution_2020}\label{shadowing lemma}
Let $n\in \bZ_{\geq 0}$ and $t>0$. Let $I_0=[-1,1]^{r(d-r)}$, $J=\prod_{i=1}^{r(d-r)}J_i$, where $J_i$ is an interval with $|J_i|\geq e^{-dnt}$ for each $i$. Let $\Psi:G/\Gamma \to \bR_+$ be a measurable function. Then 
\begin{align}\label{inequlity in shadowing lemma}
    \int_J \Psi(a_{(n+1)t}u_{\boldsymbol{\varphi}}(\mathbf{s})\Gamma)d\mathbf{s}\leq \int_J \int_{I_0}\Psi(a_{(n+1)t}u_{\boldsymbol{\varphi}}(\mathbf{s}+\Tilde{\mathbf{s}}e^{-dnt})\Gamma)d\Tilde{\mathbf{s}}d\mathbf{s}.
\end{align}
\end{lemma}

\begin{lemma}\label{consequence of good property}
Let $\kappa:I_0=[-1,1]^{r(d-r)}\to \bR_+$ be a measurable function. Suppose that there exists $C>0$ such that for any $\epsilon>0$,
\begin{align*}
    |\{\mathbf{s}\in I_0:\kappa(\mathbf{s})<\epsilon\}|\leq C\cdot \epsilon^{\frac{1}{r(d-r)}}.
\end{align*}
Then for any $0<\nu<\frac{1}{r(d-r)}$, there exists $c_{\nu}>0$ such that
\begin{align*}
    \int_{I_0}\kappa(\mathbf{s})^{-\nu}d\mathbf{s}\leq C^{\nu\cdot r(d-r)}\cdot c_{\nu}.
\end{align*}
\end{lemma}

\begin{proof}
This is a direct generalization of \cite[Lemma 6.10]{Shi_Ulcigrai_Genericity_on_curves_2018}.
\end{proof}

The following lemma is a special case of \cite[Lemma 3.3]{kleinbock1998flows}.

\begin{lemma}\label{good property of functions}
Let $V$ be a bounded open subset of $Mat_{r\times(d-r)}(\bR)$, and let $f\in C^1(V)$ be such that for some constants $A_1,A_2>0$, one has
\begin{align*}
 &A_2\leq |\partial_{ij}f(\mathbf{s})|\leq A_1 ,\forall 1\leq i\leq r, 1\leq j \leq d-r, \forall \mathbf{s}\in V,\\
    &\text{ and } \norm{f}_V\leq A_1,
\end{align*}
where $\norm{\cdot}_V$ denote the sup norm of a function on $V$. Then for any box {(or ball)} $B\subset V$, any $\epsilon>0$, one has 
\begin{align*}
    |\{\mathbf{s}\in B: |f(\mathbf{s})|<\epsilon\}|\leq r(d-r)\cdot C_{A_1,A_2}(\frac{\epsilon}{\norm{f}_B})^{\frac{1}{r(d-r)}}|B|,
\end{align*}
with $C_{A_1,A_2}=\frac{12A_1}{A_2}$.
\end{lemma}
\begin{lemma}
{Let $N\geq 1$ be an integer. Let $f:\mathbb{R}^N\to \mathbb{R}$ be a $C^1$ map. Given $\mathbf{B}\in SO_N(\mathbb{R})$, for any $\mathbf{s}\in \mathbb{R}^N$, let $\mathbf{s}\prm=\mathbf{B}\mathbf{s}$. Then 
\begin{align*}
 (\frac{\partial f}{\partial s_1\prm},\cdots,\frac{\partial f}{\partial s_N\prm})^t=\mathbf{B}(\frac{\partial f}{\partial s_1},\cdots,\frac{\partial f}{\partial s_N})^t,
\end{align*}
where superscript $t$ denotes the transpose of the vector.}
\end{lemma}
\begin{proof}
{Since $\mathbf{s}\prm=\mathbf{B}\mathbf{s}$, $\mathbf{s}=\mathbf{B}^{-1}\mathbf{s}\prm$. Using chain rule, it can be verified that 
\begin{align*}
    (\frac{\partial f}{\partial s_1\prm},\cdots,\frac{\partial f}{\partial s_N\prm})^t=(\mathbf{B}^{-1})^t \cdot (\frac{\partial f}{\partial s_1},\cdots,\frac{\partial f}{\partial s_N})^t.
\end{align*}
As $\mathbf{B}\in SO_N(\mathbb{R})$, we have $(\mathbf{B}^{-1})^t=\mathbf{B}$. This proves the lemma.
}
\end{proof}
\begin{lemma}\label{basic linear algebra lemma}
{Let $N\geq 1$ be an integer. Given real numbers $0<c_2<C_2<c_1<C_1$, and a partition $\{\mathcal{I}_1,\mathcal{I}_2,\mathcal{I}_3\}$ of $\{1,\cdots,N\}$ such that $\mathcal{I}_1\neq \emptyset$. There is $\mathbf{B}\in SO_N(\mathbb{R})$ (depending only on the partition) such that the following holds: For any vector $\mathbf{v}=(v_1,\cdots,v_N)^{t}\in \mathbb{R}^N$ (here superscript $t$ denote the transpose of the corresponding vector) satisfying
\begin{itemize}
    \item For any $i\in \{1,\cdots,N\}$, $|v_i|\leq C_1$;
    \item For any $i\in \mathcal{I}_1$, $|v_i|\geq c_1$;
    \item For any $i\in \mathcal{I}_2$, $|v_i|\geq C_2$;
    \item For any $i \in \mathcal{I}_3$, $|v_i|\leq c_2$.
    \end{itemize}
If we denote $\mathbf{v}^{\prime} =(v_1\prm,\cdots,v_N\prm)^t=\mathbf{B}\mathbf{v}$, then for any $i=1,\cdots,N$,
\begin{align*}
    \min\{\frac{c_1}{\sqrt{N}}-\sqrt{N}c_2,C_2\}\leq |v_i\prm|\leq C_1+\sqrt{N}c_2.
\end{align*}
}
\end{lemma}
\begin{proof}
{
If $\mathcal{I}_3=\emptyset$, then the lemma is trivial. Now we assume that $\mathcal{I}_3\neq \emptyset$. Let $p\in\mathbb{N}$ be such that $p-1=\# \mathcal{I}_3$, then $2\leq p\leq N$. Without loss of generality, we may assume that $1\in \mathcal{I}_1$ and $\mathcal{I}_3=\{2,\cdots,p\}$.
}
{
Choose a $\mathbf{B}=(b_{ij})\in SO_N(\mathbb{R})$ satisfying
\begin{itemize}
    \item $b_{i1}=\frac{1}{\sqrt{p}}$, for $1\leq i\leq p$;
    \item $b_{ij}=0$ if $j\leq p<i$, or $i\leq p <j$;
    \item $b_{ij}=\delta_{ij}$ if $i\geq p+1$ and $j\geq p+1$.
\end{itemize}
Note that for any $i=1,\cdots,p$, $v_i\prm=1/\sqrt{p}\cdot v_1+\sum_{j=2}^p b_{ij}v_j$. Therefore, for any $1\leq i\leq p$, we have the lower bound
\begin{align*}
    |v_i\prm|\geq \frac{1}{\sqrt{p}}|v_1|-\sum_{j=2}^p |b_{ij}||v_j|\geq \frac{c_1}{\sqrt{p}}-(\sum_{j=2}^p |c_2|^2)^{\frac{1}{2}}\geq \frac{c_1}{\sqrt{N}}-\sqrt{N}c_2,
\end{align*}
where in the second inequality we apply Cauchy-Schwartz inequality. Also for any $1\leq i\leq p$, we have the upper bound
\begin{align*}
    |v_i\prm|\leq \frac{1}{\sqrt{p}}|v_1|+\sum_{j=2}^p |b_{ij}||v_j|\leq C_1+\sqrt{N} c_2.
\end{align*}
On the other hand, for any $p+1\leq i\leq N$, we have $v_i\prm=v_i$. Therefore, for any $1\leq i\leq N$,
\begin{align*}
    \min\{\frac{c_1}{\sqrt{N}}-\sqrt{N}c_2,C_2\}\leq |v_i\prm|\leq C_1+\sqrt{N}c_2.
\end{align*}
}
\end{proof}
{Roughly speaking, Lemma \ref{basic linear algebra lemma} says that one can find a suitable rotation $\mathbf{B}\in SO_N(\mathbb{R})$ depending \textbf{only} on the partition of $\{1,\cdots,N\}$ such that for any vector $\mathbf{v}\in \mathbb{R}^N$, as long as there is a coordinate of $\mathbf{v}$ with large enough absolute value, the absolute value of all coordinates of the new vector $\mathbf{B}\mathbf{v}$ are bounded below by a suitable constant.}

\begin{proof}[Proof of (5) of Proposition \ref{contraction of mixed height functions}]

If $n=0$ and $J=I$. Then for any $\mathbf{s}\in I$, by Proposition \ref{contraction of mixed height functions} (4),
\begin{align*}
    \beta_{\mathbf{m}}(a_tu_{\boldsymbol{\varphi}}(\mathbf{s})\Gamma)
    \leq \Tilde{M}_2 \beta_{\mathbf{m}}(u_{\boldsymbol{\varphi}}(\mathbf{s})\Gamma) 
   \leq \Tilde{M}_2(\sigma^{-\nu}+M).
\end{align*}
Then for any $b\in \bR$ such that $b>\Tilde{M}_2(\sigma^{-\nu}+M)$,
\begin{align*}
    \int_I \beta_{\mathbf{m}}(a_t u_{\boldsymbol{\varphi}}(\mathbf{s})\Gamma)d\mathbf{s}\leq \frac{1}{2}\int_I\beta_{\mathbf{m}}(u_{\boldsymbol{\varphi}}(\mathbf{s})\Gamma)d\mathbf{s}+b|I|.
\end{align*}

Now we assume $n\geq 1$ and let $t>0$ be a sufficiently large number (to be specified later). By Lemma \ref{shadowing lemma},
\begin{align*}
     \int_J\beta_{\mathbf{m}}(a_{(n+1)t}u_{\boldsymbol{\varphi}}(\mathbf{s})\Gamma)d\mathbf{s}\leq \int_J\int_{I_0}\beta_{\mathbf{m}}(a_{(n+1)t}u_{\boldsymbol{\varphi}}(\mathbf{s}+\Tilde{\mathbf{s}}e^{-dnt})\Gamma)d\Tilde{\mathbf{s}}d\mathbf{s}.
\end{align*}
By Proposition \ref{contraction hypothesis for Margulis' height function}, for $t>0$ sufficiently large, there exists $b_1>0$ such that for any $\mathbf{s}\in J$,
\begin{align}\label{contraction of tilde alpha for u_varphi}
    \int_{I_0}\Tilde{\alpha}(a_t u(\Tilde{\mathbf{s}})a_{nt}u(\mathbf{s})\Gamma)d\Tilde{\mathbf{s}}
    \leq \frac{1}{4}\Tilde{\alpha}(a_{nt}u_{\boldsymbol{\varphi}}(\mathbf{s})\Gamma)+b_1.
\end{align}
Note that since
\begin{align*}
   \int_{I_0}\Tilde{\alpha}(a_t u(\Tilde{\mathbf{s}})a_{nt}u(\mathbf{s})\Gamma)d\Tilde{\mathbf{s}}&= \int_{I_0}\Tilde{\alpha}(a_{(n+1)t}u_{\boldsymbol{\varphi}}(\mathbf{s}+\Tilde{\mathbf{s}}e^{-dnt})\Gamma)d\Tilde{\mathbf{s}}\\
    &\leq \frac{1}{4}\Tilde{\alpha}(a_{nt}u_{\boldsymbol{\varphi}}(\mathbf{s})\Gamma)+b_1,
\end{align*}
by definition of $\beta_{\mathbf{m}}$, it remains to estimate the following integral for any $\mathbf{s}\in J$:
\begin{align}\label{integral of alpha m}
    \int_{I_0}\alpha_{\mathbf{m}}^{\nu}(a_{(n+1)t}u_{\boldsymbol{\varphi}}(\mathbf{s}+\Tilde{\mathbf{s}}e^{-dnt})\Gamma)d\Tilde{\mathbf{s}}.
\end{align}
Define {for any $\mathbf{s}\in J$,}
\begin{align*}
    &I_{01}(\mathbf{s}):=\{\Tilde{\mathbf{s}}\in I_0: \hat{\mathbf{s}}=\mathbf{s}+\Tilde{\mathbf{s}}e^{-dnt}, \boldsymbol{\xi}_{n+1,\hat{\mathbf{s}}} \text{ exists},\boldsymbol{\xi}_{n+1,\hat{\mathbf{s}}}\neq \boldsymbol{\xi}_{n,\mathbf{s}}\},\\
    &I_{02}(\mathbf{s}):=\{\Tilde{\mathbf{s}}\in I_0: \hat{\mathbf{s}}=\mathbf{s}+\Tilde{\mathbf{s}}e^{-dnt}, \boldsymbol{\xi}_{n+1,\hat{\mathbf{s}}} \text{ exists},\boldsymbol{\xi}_{n+1,\hat{\mathbf{s}}}= \boldsymbol{\xi}_{n,\mathbf{s}}\},\\
    &I_{03}(\mathbf{s}):=\{\Tilde{\mathbf{s}}\in I_0:\hat{\mathbf{s}}=\mathbf{s}+\Tilde{\mathbf{s}}e^{-dnt}, \boldsymbol{\xi}_{n+1,\hat{\mathbf{s}}} \text{ does not exist}\}.
\end{align*}
Since for $\tilde{\mathbf{s}}\in I_{03}$, $\alpha_{\mathbf{m}}^{\nu}$ is dominated by $\tilde{\alpha}$, by Lemmas \ref{Lemma: estimation of I_01}, \ref{Lemma:estimate of I_02} given as follows, we have
\begin{align*}
    &\int_{I_0}\alpha_{\mathbf{m}}^{\nu}(a_{(n+1)t}u_{\boldsymbol{\varphi}}(\mathbf{s}+\Tilde{\mathbf{s}}e^{-dnt})\Gamma)d\Tilde{\mathbf{s}}=\int_{I_{01}(\mathbf{s})}\alpha_{\mathbf{m}}^{\nu}d\Tilde{\mathbf{s}} +\int_{I_{02}(\mathbf{s})}\alpha_{\mathbf{m}}^{\nu}d\Tilde{\mathbf{s}} +\int_{I_{03}(\mathbf{s})}\alpha_{\mathbf{m}}^{\nu}d\Tilde{\mathbf{s}}\\
    &\leq (10r^2d)^{\nu}e^{\nu rt}\cdot 2^{r(d-r)}\cdot \Tilde{\alpha}(a_{nt}u_{\boldsymbol{\varphi}}(\mathbf{s})\Gamma)+\frac{1}{4}\alpha_{\mathbf{m}}^{\nu}(a_{nt}u_{\boldsymbol{\varphi}}(\mathbf{s})\Gamma)+2^{\nu+r(d-r)}.
\end{align*}
As we choose $b>2^{\nu+r(d-r)}+cb_1e^{r\nu t}$, recall that $c=4\cdot (10r^2d)^{\nu}\cdot 2^{r(d-r)}$, we have
\begin{align*}
     &\int_J\beta_{\mathbf{m}}(a_{(n+1)t}u_{\boldsymbol{\varphi}}(\mathbf{s})\Gamma)d\mathbf{s}\leq \int_J\int_{I_0}\beta_{\mathbf{m}}(a_{(n+1)t}u_{\boldsymbol{\varphi}}(\mathbf{s}+\Tilde{\mathbf{s}}e^{-dnt})\Gamma)d\Tilde{\mathbf{s}}d\mathbf{s}\\
     &\leq \int_J[\frac{1}{4}c\cdot e^{r\nu t}\Tilde{\alpha}(a_{nt}u_{\boldsymbol{\varphi}}(\mathbf{s})\Gamma)+\frac{1}{4}\alpha_{\mathbf{m}}^{\nu}(a_{nt}u_{\boldsymbol{\varphi}}(\mathbf{s})\Gamma)+2^{\nu+r(d-r)}+c b_1 e^{r\nu t}] d\mathbf{s},\\
     &\leq \frac{1}{2}\int_J \beta_{\mathbf{m}}(a_{nt}u_{\boldsymbol{\varphi}}(\mathbf{s})\Gamma)d\mathbf{s}+b|J|.
\end{align*}
This finishes the proof of property (5) of Proposition \ref{contraction of mixed height functions}, modulo Lemmas \ref{Lemma: estimation of I_01}, \ref{Lemma:estimate of I_02}.

\end{proof}

\begin{lemma}\label{Lemma: estimation of I_01}
Let $J$ be the box as in Proposition \ref{contraction of mixed height functions} (5). There is $t>0$ sufficiently large such that for any $\mathbf{s}\in J$,
\begin{align*}
    \int_{I_{01}(\mathbf{s})}\alpha_{\mathbf{m}}^{\nu}(a_{(n+1)t}u_{\boldsymbol{\varphi}}(\mathbf{s}+\Tilde{\mathbf{s}}e^{-dnt})\Gamma)d\Tilde{\mathbf{s}}\leq e^{\nu rt}(10r^2d)^{\nu}2^{r(d-r)}\cdot \Tilde{\alpha}(a_{nt}u_{\boldsymbol{\varphi}}(\mathbf{s})\Gamma).
\end{align*}
\end{lemma}
\begin{proof}
We will prove that for $t>0$ sufficiently large, for any $\Tilde{\mathbf{s}}\in I_{01}(\mathbf{s})$,
\begin{align*}
    \alpha_{\mathbf{m}}^{\nu}(a_{(n+1)t}u_{\boldsymbol{\varphi}}(\mathbf{s}+\Tilde{\mathbf{s}}e^{-dnt})\Gamma)\leq e^{\nu rt}(10r^2d)^{\nu}\cdot \Tilde{\alpha}(a_{nt}u_{\boldsymbol{\varphi}}(\mathbf{s})\Gamma).
\end{align*}
For $\tilde{\mathbf{s}}\in I_{0}$, denote $\hat{\mathbf{s}}=\mathbf{s}+\Tilde{\mathbf{s}}e^{-dnt}$.

\textbf{Case 1}: $\norm{(w(n,\mathbf{s},\boldsymbol{\xi}_{n+1,\hat{\mathbf{s}}}))_{\leq r}}_{\infty}\geq N_2\norm{(w(n,\mathbf{s},\boldsymbol{\xi}_{n+1,\hat{\mathbf{s}}}))_{>r}}_{\infty}$, where $N_2=2(r+1)(d-r)$. 

By definition of $M_1$ (cf.(\ref{definition of M1})), the choice of $N_1$ and Lipschity continuity of $\boldsymbol{\varphi}$, we have for any $i,j,p,q$,
\begin{align*}
    |\sum_{q=1}^k \frac{\partial \boldsymbol{\varphi}_{pq}}{\partial s_{ij}}\cdot m_q| 
    \leq \norm{\mathbf{m}}\cdot (\sum_{q=1}^k|\frac{\partial \boldsymbol{\varphi}_{pq}}{\partial s_{ij}}|^2)^{\frac{1}{2}}
    \leq \norm{\mathbf{m}}k^{\frac{1}{2}}\frac{M_1}{N_1}\leq \norm{(\boldsymbol{\xi}_{n+1,\hat{\mathbf{s}}})_{>r}\cdot \mathbf{m}}_{\infty}.
\end{align*}
By Lemma \ref{consequence of existence of xi}, the choices of the sidelength of the box and $N_2$,
\begin{align*}
    &\norm{w(n+1,\hat{\mathbf{s}},\boldsymbol{\xi}_{n+1,\hat{\mathbf{s}}})}\geq \norm{(w(n+1,\hat{\mathbf{s}},\boldsymbol{\xi}_{n+1,\hat{\mathbf{s}}})_{\leq r}}_{\infty}\\
    &=e^{(d-r)(n+1)t}\norm{[(\boldsymbol{\varphi}(\boldsymbol{\hat{\mathbf{s}}}))_{\leq r}-(\boldsymbol{\xi}_{n+1,\hat{\mathbf{s}}})_{\leq r}-\hat{\mathbf{s}}\cdot (\boldsymbol{\xi}_{n+1,\hat{\mathbf{s}}})_{>r}]\cdot \mathbf{m}}_{\infty}\\
    &\geq e^{(d-r)t}\norm{(w(n,\mathbf{s},\boldsymbol{\xi}_{n+1,\hat{\mathbf{s}}}))_{\leq r}}_{\infty}-(1-\frac{(r+1)(d-r)}{N_2}) e^{(d-r)t}\norm{(w(n,\mathbf{s},\boldsymbol{\xi}_{n+1,\hat{\mathbf{s}}}))_{\leq r}}_{\infty}\\
    &\geq  e^{(d-r)t} \frac{1}{2d} \norm{w(n,\mathbf{s},\boldsymbol{\xi}_{n+1,\hat{\mathbf{s}}})}.
\end{align*}
Choose $t>0$ large enough such that $e^{(d-r)t}\frac{1}{2d}>1$, we obtain
\begin{align*}
   \alpha_{\mathbf{m}}^{\nu}(a_{(n+1)t}u_{\boldsymbol{\varphi}}(\mathbf{s}+\Tilde{\mathbf{s}}e^{-dnt})\Gamma)&=\norm{w(n+1,\hat{\mathbf{s}},\boldsymbol{\xi}_{n+1,\hat{\mathbf{s}}})}^{-\nu}\\
   &\leq e^{-\nu(d-r)t}(2d)^{\nu}\norm{w(n,\mathbf{s},\boldsymbol{\xi}_{n+1,\hat{\mathbf{s}}})}^{-\nu}\\
    &\leq e^{-\nu(d-r)t}(2d)^{\nu}\cdot 2^{\nu} \sup_{\mathbf{w}\in \bZ^d\setminus \{\mathbf{0}\}}\norm{a_{nt}u(\mathbf{s})\mathbf{w}}^{-\nu}\\
    &\leq \Tilde{\alpha}(a_{nt}u_{\boldsymbol{\varphi}(\mathbf{s})}\Gamma).
\end{align*}

\textbf{Case 2}: $\norm{(w(n,\mathbf{s},\boldsymbol{\xi}_{n+1,\hat{\mathbf{s}}}))_{\leq r}}_{\infty}< N_2\norm{(w(n,\mathbf{s},\boldsymbol{\xi}_{n+1,\hat{\mathbf{s}}}))_{>r}}_{\infty}$.

Then by the choice of $N_2$, we have
\begin{align*}
    \norm{w(n+1,\hat{\mathbf{s}},\boldsymbol{\xi}_{n+1,\hat{\mathbf{s}}})}\geq e^{-rt}\norm{(w(n,\mathbf{s},\boldsymbol{\xi}_{n+1,\hat{\mathbf{s}}}))_{>r}}_{\infty}
    \geq e^{-rt}\frac{1}{5r^2d}\norm{w(n,\mathbf{s},\boldsymbol{\xi}_{n+1,\hat{\mathbf{s}}})}.
\end{align*}
Therefore, 
\begin{align*}
     \alpha_{\mathbf{m}}^{\nu}(a_{(n+1)t}u_{\boldsymbol{\varphi}}(\mathbf{s}+\Tilde{\mathbf{s}}e^{-dnt})\Gamma)&=\norm{w(n+1,\hat{\mathbf{s}},\boldsymbol{\xi}_{n+1,\hat{\mathbf{s}}})}^{-\nu}\\
     &\leq e^{\nu rt} (5r^2d)^{\nu}\norm{w(n,\mathbf{s},\boldsymbol{\xi}_{n+1,\hat{\mathbf{s}}})}^{-\nu}\\
     &\leq e^{\nu rt} (10r^2d)^{\nu}\sup_{\mathbf{w}\in \bZ^d\setminus \{\mathbf{0}\}}\norm{a_{nt}u(\mathbf{s})\mathbf{w}}^{-\nu}\\
     &\leq  e^{\nu rt} (10r^2d)^{\nu}\Tilde{\alpha}(a_{nt}u_{\boldsymbol{\varphi}(\mathbf{s})}\Gamma).
\end{align*}
Combining cases 1 and 2, the lemma is proven. 
\end{proof}

\begin{lemma}\label{Lemma:estimate of I_02}
There exists $t>0$ sufficiently large such that {for any $\mathbf{s}\in J$,}
\begin{align*}
   \int_{I_{02}(\mathbf{s})}\alpha_{\mathbf{m}}^{\nu}(a_{(n+1)t}u_{\boldsymbol{\varphi}}(\mathbf{s}+\Tilde{\mathbf{s}}e^{-dnt})\Gamma)d\tilde{\mathbf{s}}\leq \frac{1}{4}\alpha_{\mathbf{m}}^{\nu}(a_{nt}u_{\boldsymbol{\varphi}}(\mathbf{s})\Gamma).
\end{align*}
\end{lemma}
\begin{proof}
{We fix $\mathbf{s}\in J$ for the rest of the proof.} For $\tilde{\mathbf{s}}\in I_{0}$, denote $\hat{\mathbf{s}}=\mathbf{s}+\Tilde{\mathbf{s}}e^{-dnt}$. Since $\tilde{\mathbf{s}}\in I_{02}(\mathbf{s})$, for simplicity we denote $\boldsymbol{\xi}_{n+1,\hat{\mathbf{s}}}=\boldsymbol{\xi}_{n,\mathbf{s}}=\mathbf{v}$.

\textbf{Case 1}: $\norm{(w(n,\mathbf{s},\mathbf{v}))_{\leq r}}_{\infty}\geq N_2\norm{(w(n,\mathbf{s},\mathbf{v}))_{>r}}_{\infty}$, where $N_2=2(r+1)(d-r)$. 

Then we have by \textbf{Case 1} of Lemma \ref{Lemma: estimation of I_01},
\begin{align*}
   \norm{w(n+1,\hat{\mathbf{s}},\mathbf{v})}\geq e^{(d-r)t} \frac{1}{2d} \norm{w(n,\mathbf{s},\mathbf{v})}.
\end{align*}
Hence, for $t>0$ sufficiently large such that $e^{-(d-r)\nu t}\cdot (2d)^{\nu}\leq \frac{1}{4}$, we obtain
\begin{align*}
    \alpha_{\mathbf{m}}^{\nu}(a_{(n+1)t}u_{\boldsymbol{\varphi}}(\mathbf{s}+\Tilde{\mathbf{s}}e^{-dnt})\Gamma)\leq \frac{1}{4}\alpha_{\mathbf{m}}^{\nu}(a_{nt}u_{\boldsymbol{\varphi}}(\mathbf{s})\Gamma).
\end{align*}

\textbf{Case 2}: $\norm{(w(n,\mathbf{s},\mathbf{v}))_{\leq r}}_{\infty}< N_2\norm{(w(n,\mathbf{s},\mathbf{v}))_{>r}}_{\infty}$.

{Recall that $I_0=[-1,1]^{r(d-r)}$. For any $\mathbf{B}\in SO_{r(d-r)}(\mathbb{R})$, we have $\mathbf{B}\cdot I_0\subset I_0\prm$, where $I_0\prm$ is the unit ball in $\mathbb{R}^{r(d-r)}$. We may choose $t>0$ large enough such that for any $\mathbf{s}\in I$, any $\tilde{\mathbf{s}}\in I_0\prm$, we have $\mathbf{s}+e^{-dnt}\tilde{\mathbf{s}}\in I\prm$, where $I\prm$ is given as in Remark \ref{remark: a small neighborhood of I with inf>0}.} Define a function $S$ on $I_0\prm$ by 
\begin{align*}
    S(\Tilde{\mathbf{s}})=\sum_{i=1}^r w_i(n+1,\mathbf{s}+\tilde{\mathbf{s}}e^{-dnt},\mathbf{v}),\forall \tilde{\mathbf{s}}\in I_0\prm.
\end{align*}
Note that $\norm{w(n+1,\hat{\mathbf{s}},\mathbf{v})}\geq \frac{1}{r} |S(\Tilde{\mathbf{s}})|$.

{We will apply Lemma \ref{basic linear algebra lemma} to find $\mathbf{B}\in SO_{r(d-r)}(\mathbb{R})$ such that after the change of basis  $\tilde{\mathbf{s}}^{\prime}=\mathbf{B}\tilde{\mathbf{s}}$, for
\begin{align}\label{align:choices of A_1,A_2}
    &A_1=e^{(d-r)t}\max\{4\sqrt{r(d-r)},2(r(d-r))^{3/2}+r\}\norm{w(n,\mathbf{s},\mathbf{v})},\nonumber\\
    &A_2=e^{(d-r)t}\frac{1}{40r^3 d(d-r)}\norm{w(n,\mathbf{s},\mathbf{v})},
\end{align}
we have
\begin{align}\label{key estimate}
    A_2\leq \norm{S}_{I_0\prm}\leq A_1;\text{  }
    A_2\leq |\frac{\partial S}{\partial \tilde{s}_{ij}\prm}(\tilde{\mathbf{s}}\prm)|\leq A_1, \forall i,j, \forall \tilde{\mathbf{s}}^{\prime}\in  I_0\prm.
\end{align}
}
Applying Lemma \ref{good property of functions} to $S(\Tilde{\mathbf{s}}\prm)$, since $\mathbf{B}$ preserves Lebesgue measure, we obtain that for any $\epsilon>0$,
\begin{align*}
    |\{\Tilde{\mathbf{s}}\in I_0:|S(\Tilde{\mathbf{s}})|\leq \epsilon\}|&\leq |\{\Tilde{\mathbf{s}}\prm\in I_0\prm:|S(\Tilde{\mathbf{s}}\prm)|\leq \epsilon\}|\\
    &\leq r(d-r)\frac{12A_1}{A_2}\cdot (\frac{\epsilon}{\norm{S}}_{I_0\prm})^{\frac{1}{r(d-r)}}\cdot |I_0\prm|\\
    &\leq \Tilde{C}\cdot e^{-\frac{1}{r}t}\cdot \epsilon^{\frac{1}{r(d-r)}} \cdot \norm{w(n,\mathbf{s},\mathbf{v})}^{-\frac{1}{r(d-r)}},
\end{align*}
where $\Tilde{C}$ is a constant depending only on $r$ and $d$. Choose $t>0$ large enough, by Lemma \ref{consequence of good property}, with $0<\nu<\frac{1}{r(d-r)}$, 
\begin{align*}
     \int_{I_{02}(\mathbf{s})}\alpha_{\mathbf{m}}^{\nu}(a_{(n+1)t}u_{\boldsymbol{\varphi}}(\mathbf{s}+\Tilde{\mathbf{s}}e^{-dnt})\Gamma)d\tilde{\mathbf{s}}&\leq \int_{I_0}\frac{1}{\norm{w(n+1,\mathbf{s}+e^{-dnt}\tilde{\mathbf{s}},\mathbf{v})}^{\nu}}d\Tilde{\mathbf{s}}\\
     &\leq r^{\nu}\cdot \int_{I_0} \frac{1}{|S(\Tilde{\mathbf{s}})|^{\nu}}d\Tilde{\mathbf{s}}\\
    &\leq r^{\nu} c_{\nu}\Tilde{C}^{\nu r(d-r)}\cdot e^{-\nu (d-r)t}\norm{w(n,\mathbf{s},\mathbf{v})}^{-\nu}\\
    &\leq \frac{1}{4}\alpha_{\mathbf{m}}^{\nu}(a_{nt}u_{\boldsymbol{\varphi}}(\mathbf{s})\Gamma).
\end{align*}
This prove the lemma. Therefore, it remains to achieve (\ref{key estimate}). Consider the function $\Psi=\sum_{q=1}^r \sum_{p=1}^k m_p \boldsymbol{\varphi}_{qp} $, where $\mathbf{m}=(m_1,\cdots,m_k)^t$. We have
\begin{align}
    S(\Tilde{\mathbf{s}})=
    e^{(d-r)(n+1)t}&[\Psi(\mathbf{s}+e^{-dnt}\tilde{\mathbf{s}})-\Psi(\mathbf{s})- \sum_{i=1}^r\sum_{j=1}^{d-r}e^{-dnt}\Tilde{s}_{ij}<(\mathbf{v})_{>r}\cdot\mathbf{m}, \mathbf{e}_j> \nonumber \\
    &+e^{-(d-r)nt}\sum_{i=1}^r w_i(n,\mathbf{s},\mathbf{v})].\label{align:expression for S}
\end{align}
{Let $N_3=4r(d-r)$, define the partition $\{\mathcal{I}_1(\mathbf{s}),\mathcal{I}_2(\mathbf{s}),\mathcal{I}_3(\mathbf{s})\}$ of $\{(i,j):1\leq i\leq r, 1\leq j \leq d-r\}$ by
\begin{align*}
    \mathcal{I}_1(\mathbf{s}):=\{(i,j):\left|<(\mathbf{v})_{>r}\cdot\mathbf{m}, \mathbf{e}_j> \right|=\norm{(\mathbf{v})_{>r}\cdot\mathbf{m}}_{\infty}\};\\
    {\mathcal{I}_2(\mathbf{s}):=\{(i,j):\norm{(\mathbf{v})_{>r}\cdot\mathbf{m}}_{\infty}>\left|<(\mathbf{v})_{>r}\cdot\mathbf{m}, \mathbf{e}_j>\right|\geq \frac{1}{N_3}\norm{(\mathbf{v})_{>r}\cdot\mathbf{m}}_{\infty}\};}\\
    \mathcal{I}_3(\mathbf{s}):=\{(i,j):\left|<(\mathbf{v})_{>r}\cdot\mathbf{m}, \mathbf{e}_j>\right|<\frac{1}{N_3}\norm{(\mathbf{v})_{>r}\cdot\mathbf{m}}_{\infty}\}.
\end{align*}
Note that by definition, $\mathcal{I}_1(\mathbf{s})\neq \emptyset$. Using (\ref{align:expression for S}), the choice of $N_3$, and the estimate
\begin{align*}
     {|\frac{\partial \Psi}{\partial \tilde{s}_{ij}}(\mathbf{s}+e^{-dnt}\tilde{\mathbf{s}})|}&\leq (\sum_{p=1}^k m_p^2)^{\frac{1}{2}}\cdot (\sum_{p=1}^k(\sum_{q=1}^r \frac{\partial \varphi_{qp}}{\partial s_{ij}})^2))^{\frac{1}{2}}
    \leq\norm{\mathbf{m}} \frac{k^{\frac{1}{2}}\cdot rM_1}{N_1}, \forall \tilde{\mathbf{s}}\in I_0\prm,
\end{align*}
the following holds for any $\tilde{\mathbf{s}}\in I_0\prm$:
\begin{itemize}
\item For any $(i,j)$, where $1\leq i\leq r$, $1\leq j\leq d-r$,
\begin{align*}
    |\frac{\partial S}{\partial \tilde{s}_{ij}}(\tilde{\mathbf{s}})|\leq C_1:=2e^{(d-r)t}\norm{(w(n,\mathbf{s},\mathbf{v}))_{>r}}_{\infty};
\end{align*}
\item For any $(i,j)\in \mathcal{I}_1(\mathbf{s})$, 
\begin{align*}
     |\frac{\partial S}{\partial \tilde{s}_{ij}}(\tilde{\mathbf{s}})|\geq c_1:=e^{(d-r)t}(1-\frac{k^{1/2}r}{N_1})\norm{(w(n,\mathbf{s},\mathbf{v}))_{>r}}_{\infty};
\end{align*}
\item For any $(i,j)\in \mathcal{I}_2(\mathbf{s})$, 
\begin{align*}
  |\frac{\partial S}{\partial \tilde{s}_{ij}}(\tilde{\mathbf{s}})|>C_2:=e^{(d-r)t}(\frac{1}{N_3}-\frac{k^{1/2}r}{N_1})\norm{(w(n,\mathbf{s},\mathbf{v}))_{>r}}_{\infty};
\end{align*}
\item  For any $(i,j)\in \mathcal{I}_3(\mathbf{s})$,
\begin{align*}
    |\frac{\partial S}{\partial \tilde{s}_{ij}}(\tilde{\mathbf{s}})|<c_2:= e^{(d-r)t}(\frac{1}{N_3}+\frac{k^{1/2}r}{N_1})\norm{(w(n,\mathbf{s},\mathbf{v}))_{>r}}_{\infty}.
\end{align*}
\end{itemize}
}
Applying Lemma \ref{basic linear algebra lemma} with $N=r(d-r)$, and $C_1,c_1,C_2,c_2$ given as above, we obtain $\mathbf{B}=\mathbf{B}(\mathbf{s})\in SO_{r(d-r)}(\bR)$ (Since the partition depends only on $\mathbf{s}$, $\mathbf{B}$ depends only on $\mathbf{s}$), such that after the change of basis $\mathbf{\tilde{s}\prm}=\mathbf{B}\mathbf{\Tilde{s}}$, the vector $\mathbf{v}(\tilde{\mathbf{s}}\prm)=(\frac{\partial S}{\partial \tilde{s}_{ij}\prm}(\tilde{\mathbf{s}}\prm))_{ij}$ satisfies the following: For any $ 1\leq i\leq r,1\leq j\leq d-r$,
\begin{align}\label{inequality for parital prm S}
  \min\{\frac{c_1}{\sqrt{N}}-\sqrt{N}c_2,C_2\} \leq  |\frac{\partial S}{\partial \tilde{s}_{ij}\prm}(\tilde{\mathbf{s}}\prm)|\leq C_1+\sqrt{N}c_2, \forall \tilde{\mathbf{s}}\prm\in I_0\prm.
\end{align}
By the assumption of Case 2, and the choice of $N_2$, it is elementary to verify that
\begin{align}\label{condition 2}
    &\min\{\frac{c_1}{\sqrt{N}}-\sqrt{N}c_2,C_2\}\geq e^{(d-r)t}\frac{1}{40r^3d(d-r)}\norm{w(n,\mathbf{s},\mathbf{v})} ,\text{ and }\nonumber\\
& C_1+\sqrt{N}c_2 \leq 4 \sqrt{r(d-r)}e^{(d-r)t}\norm{w(n,\mathbf{s},\mathbf{v})}.
\end{align}
Moreover, using the expression (\ref{align:expression for S}), we obtain
\begin{align}\label{condition 3}
    \norm{S}_{I_0\prm}\leq e^{(d-r)t}(2(r(d-r))^{3/2}+r)\norm{w(n,\mathbf{s},\mathbf{v})}.
\end{align}
Also, note that 
\begin{align}\label{condition 4}
    \norm{S}_{I_0\prm}&
    \geq \inf_{\Tilde{\mathbf{s}}\in I_0, (i,j)\in \mathcal{I}_1}|\frac{\partial S}{\partial \tilde{s}_{ij}}(\Tilde{\mathbf{s}})|
    \geq e^{(d-r)t}\frac{1}{10r^2 d}\norm{w(n,\mathbf{s},\mathbf{v})}.
\end{align}
Now we choose $A_1,A_2$ as in (\ref{align:choices of A_1,A_2}),
by (\ref{inequality for parital prm S})(\ref{condition 2})(\ref{condition 3})(\ref{condition 4}), (\ref{key estimate}) is achieved. This finishes the proof of the lemma.
\end{proof}

\section{Proof of Proposition \ref{limit measure spend very little measure on singular sets}}\label{section proof of main proposition}
Following a general strategy developed in \cite[Section 6.6]{Shi_Ulcigrai_Genericity_on_curves_2018}, we derive Proposition \ref{limit measure spend very little measure on singular sets} from Proposition \ref{contraction of mixed height functions}. 

Let $Y$ be a locally compact, second countable Hausdorff topological space. Let $B$ be a compact box in $Mat_{r\times(d-r)}(\bR)$. Let $\phi:Mat_{r\times(d-r)}(\bR)\to Y$ be a continuous map. Let $f:\bR\times Y\to Y$ be a continuous map and we write $f(t,y)$ as $f^t(y)$ for $(t,y)\in \bR\times Y$.

Let $\mathcal{F}_0=\{B\}$. For every $n\in \mathbb{N}$, let $\mathcal{F}_n$ be a partition of elements in $\mathcal{F}_{n-1}$ into countably many subboxes with positive Lebesgue measure.
By construction, $\{\mathcal{F}_n\}_{n\in \mathbb{Z}_{\geq 0}}$ is a filtration.
For any $\mathbf{s}\in B$, let $I_n(\mathbf{s})$ denote the atom in $\mathcal{F}_n$ containing $\mathbf{s}$.

Let $\beta:Y\to [1,\infty]$ be a measurable map. Assume that $\beta$ satisfies the following conditions:

(1) $\beta$ satisfies contraction hypothesis: There exist $0<a<1$ and $b>0$ such that for any $n\in \bZ_{\geq 0}$ and any atom $I_n$ in $\mathcal{F}_n$,
\begin{align}\label{uniform contraction inequality}
    \int_{I_n}\beta(f^{n+1} \phi(\mathbf{s}))d\mathbf{s}<a\int_{I_n}\beta(f^n\phi(\mathbf{s}))d\mathbf{s}+b|I_n|;
\end{align}

(2) $\beta$ satisfies Lipschitz property: There exists a constant $M>0$ such that for any $\mathbf{s}\in B$, any $n\in \bZ_{\geq 0}$, and any $\Tilde{\mathbf{s}}\in I_n(\mathbf{s})$,
\begin{align}
    \beta(f^n\phi(\Tilde{\mathbf{s}}))\leq M\beta(f^n\phi(\mathbf{s})),\nonumber\\
    \beta(f^{n+1}\phi(\mathbf{s}))\leq M \beta(f^n\phi(\mathbf{s}))\label{Lipschitz property of beta};
\end{align}

(3) $\beta$ is bounded on $\phi(B)$, that is, there exists $l>0$ such that 
\begin{align}\label{boundedness condition of beta}
    \{\phi(\mathbf{s}): \mathbf{s}\in B\}\subset Y_l=\{y\in Y: \beta(y)<l\}.
\end{align}
For any $T>0$ and a measurable subset $K$ of $Y$, define 
\begin{align*}
    \mathcal{A}_K^T(\mathbf{s}):=\frac{1}{T}\int_0^T \chi_K(f^t\phi(\mathbf{s}))dt,
\end{align*}
where $\chi_K$ is the indicator function of $K$.

\begin{lemma}\cite[Lemma 6.20]{Shi_Ulcigrai_Genericity_on_curves_2018}\label{lemma: continuous version of consequence of small gaps of return times}
For any $\epsilon>0$, there exist $0<l_1<\infty$ and $0<c_1<1$ such that for $K=Y_{l_1}$, and any $T>1$,
\begin{align*}
    |\{\mathbf{s}\in B:\mathcal{A}_K^T(\mathbf{s})\leq 1-\epsilon\}|\leq c_1^T|B|.
\end{align*}
\end{lemma}

\begin{proof}[proof of Proposition \ref{limit measure spend very little measure on singular sets}]
We will apply Lemma \ref{lemma: continuous version of consequence of small gaps of return times} to $Y=X$, $B=I$, $\beta=\beta_{\mathbf{m}}$, $\phi(\mathbf{s})=u_{\boldsymbol{\varphi}}(\mathbf{s})\Gamma$ and $f^t=a_t$ for $t>0$ sufficiently large so that Proposition \ref{contraction of mixed height functions} holds.

Recall that $I$ is a closed cube in $Mat_{r\times(d-r)}(\bR)$. We may assume that $t>0$ is large enough such that $e^{-dt}$ is less than the length of each side of $I$.

We construct a filtration $\{\mathcal{F}_n\}_{n\in \mathbb{N}}$ on $I$ as follows. Let $\mathcal{F}_0=\{\emptyset,I\}$. Suppose that we have already constructed $\mathcal{F}_{n-1}$. We divide each box $J$ of $\mathcal{F}_{n-1}$ consecutively into cubes and boxes such that cubes have side length $e^{-dnt}$ and boxes have side length between $e^{-dnt}$ and $2e^{-dnt}$. Then conditions (\ref{uniform contraction inequality})(\ref{Lipschitz property of beta})(\ref{boundedness condition of beta}) follows from Proposition \ref{contraction of mixed height functions}.

Therefore, applying Lemma \ref{lemma: continuous version of consequence of small gaps of return times}, we obtain $l_1>0$ such that the set $K$ defined by
\begin{align*}
    K:=\{x\in X: \beta_{\mathbf{m}}(x)<l_1\}
\end{align*}
is a compact subset of $X\setminus X_{\mathbf{m}}$, and (\ref{exponential decay of average operator}) holds.
\end{proof}

\section{Proof of variants of Theorem \ref{genericity for special trajectories}}\label{section: more general version of theorem genericity condition for special}
Given $\mathbf{M}\in SL_d(\bR)$,
we may write 
\begin{align}\label{notation:block matrix}
    \mathbf{M}=\begin{bmatrix}
        \mathbf{A} &\mathbf{B}\\
        \mathbf{C} &\mathbf{D}
    \end{bmatrix},
\end{align}
where $\mathbf{A}\in Mat_{r\times r}(\bR)$, $\mathbf{B}\in Mat_{r\times (d-r)}(\bR)$, $\mathbf{C}\in Mat_{(d-r)\times r}(\bR)$, $\mathbf{D}\in Mat_{(d-r)\times (d-r)}(\bR)$. For $\mathbf{s}\in \mathcal{U}$, we can write
\begin{align}\label{align:expression for u(s)M}
    u(\mathbf{s}) \mathbf{M}=\begin{bmatrix}
        \mathbf{A}+\mathbf{s}\cdot \mathbf{C} & \mathbf{B}+\mathbf{s}\cdot \mathbf{D}\\
        \mathbf{C}&\mathbf{D}
    \end{bmatrix}
    =
    \begin{bmatrix}
        \mathbf{A}(\mathbf{s}) & \mathbf{B}(\mathbf{s})\\
        \mathbf{C} & \mathbf{D}
    \end{bmatrix}.
\end{align}
Since $\mathbf{M}\in SL_d(\bR)$, it is clear that the set of $\mathbf{s}\in \mathcal{U}$ such that $\det \mathbf{A}(\mathbf{s})=0$ is a proper algebraic subvariety of $\mathcal{U}$ and hence, it has Lebesgue measure zero.

Therefore we can assume that $\det \mathbf{A}(\mathbf{s})\neq 0$, and
\begin{align*}
    u(\mathbf{s}) \mathbf{M}
    =\begin{bmatrix}
        \mathbf{A}(\mathbf{s}) & \mathbf{0}\\
        \mathbf{C} & \mathbf{D}-\mathbf{C}\mathbf{A}(\mathbf{s})^{-1}\mathbf{B}(\mathbf{s})
    \end{bmatrix}
    \cdot \begin{bmatrix}
        \mathbf{1}_{r} & \mathbf{A}(\mathbf{s})^{-1}\mathbf{B}(\mathbf{s})\\
        \mathbf{0}_{d-r,r} & \mathbf{1}_{d-r}
    \end{bmatrix}.
\end{align*}
We may write
\begin{align*}
    u(\mathbf{s}) \mathbf{M} (Id,\boldsymbol{\varphi}(\mathbf{s}))
    =(u(\mathbf{s}) \mathbf{M} (Id,\boldsymbol{\varphi}(\mathbf{s})))_{-}
    \cdot (u(\mathbf{s}) \mathbf{M} (Id,\boldsymbol{\varphi}(\mathbf{s})))_{+},
\end{align*}
where
\begin{align*}
    &(u(\mathbf{s}) \mathbf{M} (Id,\boldsymbol{\varphi}(\mathbf{s})))_{-}=\begin{bmatrix}
        \mathbf{A}(\mathbf{s}) & \mathbf{0}\\
        \mathbf{C} & \mathbf{D}-\mathbf{C}\mathbf{A}^{-1}(\mathbf{s})\mathbf{B}(\mathbf{s})
    \end{bmatrix}
    \cdot
    \left(Id,\begin{bmatrix}
        \mathbf{0}\\
        (\boldsymbol{\varphi}(\mathbf{s}))_{>r}
    \end{bmatrix}\right),\\
    &(u(\mathbf{s}) \mathbf{M} (Id,\boldsymbol{\varphi}(\mathbf{s})))_{+}=\left(\begin{bmatrix}
        \mathbf{1}_{r} & \mathbf{A}(\mathbf{s})^{-1}\mathbf{B}(\mathbf{s})\\
        \mathbf{0}_{d-r,r} & \mathbf{1}_{d-r}
    \end{bmatrix}, \begin{bmatrix}
        (\boldsymbol{\varphi}(\mathbf{s}))_{\leq r}+\mathbf{A}(\mathbf{s})^{-1}\mathbf{B}(\mathbf{s})\cdot (\boldsymbol{\varphi}(\mathbf{s}))_{>r}\\
        \mathbf{0}
    \end{bmatrix}\right).
\end{align*}

\begin{lemma}\label{Lemma:difermorphism}
For a.e. $\mathbf{s}_0\in \mathcal{U}$, there is an open neighborhood $\mathcal{V}$ of $\mathbf{s}_0$ contained in $\mathcal{U}$ and an open subset $\Tilde{\mathcal{V}}$ of $Mat_{r\times(d-r)}(\bR)$, such that the map $\phi:\mathcal{V}\to \Tilde{\mathcal{V}}$ defined by
\begin{align}\label{diffeomorphism between s and tilde s}
    \phi(\mathbf{s})=\mathbf{A}(\mathbf{s})^{-1}\mathbf{B}(\mathbf{s})
\end{align}
is a diffeomorphism.
\end{lemma}

\begin{proof}
For any $\mathbf{s}_0\in \mathcal{U}$ such that $\det \mathbf{A}(\mathbf{s}_0)\neq 0$, there is a neighborhood $\mathcal{V}$ of $\mathbf{s}_0$, for any $\mathbf{s}\in \mathcal{V}$, the map $\phi(\mathbf{s})$
is well defined and differentiable. Therefore $\Tilde{\mathcal{V}}=\phi(\mathcal{V})\subset Mat_{r\times (d-r)}(\bR)$ is an open subset. Let  
\begin{align*}
    \phi^{-1}(\Tilde{\mathbf{s}})=(\mathbf{A}\Tilde{\mathbf{s}}-\mathbf{B})(\mathbf{D}-\mathbf{C}\Tilde{\mathbf{s}})^{-1}
\end{align*}
{ 
for any $\tilde{\mathbf{s}}$ such that $\det(\mathbf{D}-\mathbf{C}\tilde{\mathbf{s}})\neq 0$. Now we verify that $\phi^{-1}$ is the inverse of $\phi$, that is, we need to verify that for $\tilde{\mathbf{s}}=\phi(\mathbf{s})$, 
\begin{align}\label{align: equation in Lemma 8.1}
    \mathbf{A} \tilde{\mathbf{s}}-\mathbf{B}=\mathbf{s}(\mathbf{D}-\mathbf{C}\tilde{\mathbf{s}}).
\end{align}
}
{Note that as $\tilde{\mathbf{s}}=\phi(\mathbf{s})$, we have $(\mathbf{A}+\mathbf{s}\mathbf{C})\tilde{\mathbf{s}}=\mathbf{B}+\mathbf{s}\mathbf{D}.$}
{Therefore, left hand side of (\ref{align: equation in Lemma 8.1}) is $\mathbf{A} \tilde{\mathbf{s}}-\mathbf{B}=(\mathbf{A}+\mathbf{s}\mathbf{C}-\mathbf{s}\mathbf{C})\tilde{\mathbf{s}}-\mathbf{B}
     =\mathbf{s}\mathbf{D}-\mathbf{s}\mathbf{C}\tilde{\mathbf{s}}$,
which is equal to the right hand side of (\ref{align: equation in Lemma 8.1}).
}
\end{proof}

\begin{proof}[Proof of Theorem \ref{theorem:genericity for general trajectory}]
Choose $\mathbf{s}_0, \mathcal{V},\Tilde{\mathcal{V}}$ satisfying Lemma \ref{Lemma:difermorphism}. 
By Lemma \ref{two asymptotic parallel trajectories converge to the same}, it suffices to prove that for a.e. $\mathbf{s}\in \mathcal{V}$, the point
\begin{align}\label{general trajectory}
    (u(\mathbf{s}) \mathbf{M} (Id,\boldsymbol{\varphi}(\mathbf{s})))_{+}\Gamma
\end{align}
is Birkhoff generic with respect to $(X,\mu_X,a_t)$. For $\Tilde{\mathbf{s}}\in \Tilde{\mathcal{V}}$,
define 
\begin{align*}
    \Tilde{\boldsymbol{\varphi}}(\Tilde{\mathbf{s}}):=\begin{bmatrix}
        (\boldsymbol{\varphi}(\phi^{-1}(\Tilde{\mathbf{s}})))_{\leq r}+\Tilde{\mathbf{s}}\cdot (\boldsymbol{\varphi}(\phi^{-1}(\Tilde{\mathbf{s}})))_{>r}\\
        \mathbf{0}
    \end{bmatrix}.
\end{align*}

Applying Corollary \ref{corollary genericity for any c1 varphi} to $u_{\Tilde{\boldsymbol{\varphi}}}(\Tilde{\mathbf{s}})\Gamma$ for $\Tilde{\mathbf{s}}\in \Tilde{\mathcal{V}}$, we obtain that if for any $\mathbf{m}\in \bZ^k\setminus \{\mathbf{0}\}$,
\begin{align}\label{condition on varphi(phi)}
    |\{\Tilde{\mathbf{s}}\in \Tilde{\mathcal{V}}:(\Tilde{\boldsymbol{\varphi}}(\Tilde{\mathbf{s}}))_{\leq r}\cdot \mathbf{m}\in \Tilde{\mathbf{s}}\cdot \bZ^{d-r}+\bZ^r\}|=0,
\end{align}
then for a.e. $\Tilde{\mathbf{s}}\in \Tilde{\mathcal{V}}$, $u_{\Tilde{\boldsymbol{\varphi}}}(\Tilde{\mathbf{s}})\Gamma$ is  Birkhoff generic with respect to $(X,\mu_X,a_t)$.

{
Suppose for some $\Tilde{\mathbf{s}}\in \Tilde{\mathcal{V}}$, and some $\mathbf{m}\in \bZ^{k}\setminus\{\mathbf{0}\}$,
\begin{align}\label{condition for tilde s}
    (\Tilde{\boldsymbol{\varphi}}(\Tilde{\mathbf{s}}))_{\leq r}\cdot \mathbf{m}\in \Tilde{\mathbf{s}}\cdot \bZ^{d-r}+\bZ^r.
\end{align}
Let $\mathbf{s}=\phi^{-1}(\tilde{\mathbf{s}})$. By definition of $\phi$ and $\tilde{\boldsymbol{\varphi}}$, (\ref{condition for tilde s}) implies that there exist $\mathbf{a}\in \bZ^{d-r}$ and $\mathbf{b}\in \bZ^r$ such that 
\begin{align*}
    (\mathbf{A}(\mathbf{s})\cdot(\boldsymbol{\varphi}(\mathbf{s}))_{\leq r}+\mathbf{B}(\mathbf{s})\cdot (\boldsymbol{\varphi}(\mathbf{s}))_{>r})\cdot \mathbf{m}=\mathbf{B}(\mathbf{s})\cdot \mathbf{a}+\mathbf{A}(\mathbf{s})\cdot \mathbf{b},
\end{align*} 
then (\ref{condition for tilde s}) implies that
\begin{align}\label{equivalent condition in matrix form}
    \begin{bmatrix}
        \mathbf{A}(\mathbf{s}) & \mathbf{B}(\mathbf{s})\\
        \mathbf{C} & \mathbf{D}
    \end{bmatrix}
    \cdot \boldsymbol{\varphi}(\mathbf{s}) \cdot \mathbf{m}=
    \begin{bmatrix}
        \mathbf{B}(\mathbf{s})\cdot \mathbf{a}+\mathbf{A}(\mathbf{s})\cdot\mathbf{b}\\
        (\mathbf{C}\cdot(\boldsymbol{\varphi}(\mathbf{s}))_{\leq r}+\mathbf{D}\cdot (\boldsymbol{\varphi}(\mathbf{s}))_{>r})\cdot \mathbf{m}
    \end{bmatrix}.
\end{align}
By (\ref{align:expression for u(s)M}), (\ref{equivalent condition in matrix form}) further implies that
\begin{align*}
    \boldsymbol{\varphi}(\mathbf{s})\cdot \mathbf{m}&=\mathbf{M}^{-1}u(-\mathbf{s})\cdot  \begin{bmatrix}
        \mathbf{B}(\mathbf{s})\cdot \mathbf{a}+\mathbf{A}(\mathbf{s})\cdot\mathbf{b}\\
        (\mathbf{C}\cdot(\boldsymbol{\varphi}(\mathbf{s}))_{\leq r}+\mathbf{D}\cdot (\boldsymbol{\varphi}(\mathbf{s}))_{>r})\cdot \mathbf{m}
    \end{bmatrix} \\
    &\in 
 \mathbb{Z}^d+\mathbf{M}^{-1}u(-\mathbf{s})\cdot \begin{bmatrix}
        \mathbf{0}\\
        \bR^{d-r}
    \end{bmatrix}.
\end{align*}
}
Therefore, the condition (\ref{genericity condition for general trajectory}) implies (\ref{condition on varphi(phi)}). By definition, for any $\mathbf{s}\in \mathcal{V}$,
\begin{align*}
     (u(\mathbf{s}) \mathbf{M} (Id,\boldsymbol{\varphi}(\mathbf{s})))_{+}=u_{\Tilde{\boldsymbol{\varphi}}}(\tilde{\mathbf{s}}), \text{ where } \tilde{\mathbf{s}}=\phi(\mathbf{s}).
\end{align*}
This finishes the proof.
\end{proof}
\begin{proof}[Proof of Corollary \ref{Corollary: genericity for general base point after u_varphi}]
{
Note that $u_{\boldsymbol{\varphi}}(\mathbf{s})\cdot(\mathbf{M},\mathbf{v})=u(\mathbf{s})\mathbf{M}(Id,\tilde{\boldsymbol{\varphi}}(\mathbf{s}))$, where $\tilde{\boldsymbol{\varphi}}(\mathbf{s})=\mathbf{M^{-1}}(\boldsymbol{\varphi}(\mathbf{s})+\mathbf{v}).$ Applying Theorem \ref{theorem:genericity for general trajectory} to $u(\mathbf{s})\mathbf{M}(Id,\tilde{\boldsymbol{\varphi}}(\mathbf{s}))\Gamma$, we obtain that if for any $\mathbf{m}\in \mathbb{Z}\setminus \{\mathbf{0}\}$,
\begin{align}\label{align:condition in the proof of Corollary:genericity for general base point after u_varphi}
       |\{\mathbf{s}\in \mathcal{U}: \tilde{\boldsymbol{\varphi}}(\mathbf{s})\cdot \mathbf{m}\in \mathbf{M}^{-1}u(-\mathbf{s})\cdot \begin{bmatrix}
        \mathbf{0}\\
        \bR^{d-r}
    \end{bmatrix}
    +\mathbb{Z}^d\}|=0,
\end{align}
then for Lebesgue a.e. $\mathbf{s}\in \mathcal{U}$, $u(\mathbf{s})\mathbf{M}(Id,\tilde{\boldsymbol{\varphi}}(\mathbf{s}))\Gamma$ is Birkhoff generic with respect to $(X,\mu_X,a_t)$. By definition of $\tilde{\boldsymbol{\varphi}}$, (\ref{align:condition in the proof of Corollary:genericity for general base point after u_varphi}) is equivalent to 
\begin{align*}
    |\{\mathbf{s}\in \mathcal{U}: (\boldsymbol{\varphi}(\mathbf{s})+\mathbf{v}) \cdot \mathbf{m}\in u(-\mathbf{s})\cdot \begin{bmatrix}
        \mathbf{0}\\
        \bR^{d-r}
    \end{bmatrix}
    +\mathbf{M}\cdot \mathbb{Z}^d\}|=0.
\end{align*}
The corollary is proven,
}
\end{proof}

\begin{proof}[Proof of Corollary \ref{genericity on orbit of maximal compact group}]
Fix an $\mathbf{s}_0\in \mathcal{U}$ at which the map $\mathbf{s}\mapsto \mathbf{E_1}(\mathbf{s})^{-1}\cdot [\mathbf{e}_{r+1},\cdots,\mathbf{e}_d]$ has a nonsingular differential. 
It is enough to prove the corollary for a.e. $\mathbf{s}$ in a neighborhood of $\mathbf{s}_0$. Choose a neighborhood $\mathcal{V}$ of $\mathbf{s}_0$ such that for any $\mathbf{s}\in \mathcal{V}$, as in (\ref{notation:block matrix}) we can write
\begin{align*}
    \mathbf{E}_2(\mathbf{s})=\mathbf{E}_1(\mathbf{s})\cdot \mathbf{E}_1(\mathbf{s}_0)^{-1}=\begin{bmatrix}
        \mathbf{A}(\mathbf{s}) & \mathbf{B}(\mathbf{s})\\
        \mathbf{C}(\mathbf{s}) & \mathbf{D}(\mathbf{s})
    \end{bmatrix},
\end{align*}
where $\det \mathbf{A}(\mathbf{s})\neq 0$ and $\det \mathbf{D}(\mathbf{s})\neq 0$. This can be done by smoothness of $\mathbf{E}_1$.

Since $\mathbf{E}_2(\mathbf{s})\in SO_d(\bR)$, $\mathbf{E}_2(\mathbf{s})\cdot \mathbf{E}_2(\mathbf{s})^{t}=Id$, that is,
\begin{align*}
    \begin{bmatrix}
        \mathbf{A}(\mathbf{s})& \mathbf{B}(\mathbf{s})\\
        \mathbf{C}(\mathbf{s})& \mathbf{D}(\mathbf{s}).
    \end{bmatrix}
    \cdot \begin{bmatrix}
        \mathbf{A}(\mathbf{s})^{t} & \mathbf{C}(\mathbf{s})^{t}\\
        \mathbf{B}(\mathbf{s})^{t} & \mathbf{D}(\mathbf{s})^{t}
    \end{bmatrix}
    &=Id.
\end{align*}
In particular, we have 
\begin{align*}
    \mathbf{A}(\mathbf{s})\cdot \mathbf{C}(\mathbf{s})^{t}+\mathbf{B}(\mathbf{s})\cdot \mathbf{D}(\mathbf{s})^{t}=\mathbf{0}.
\end{align*}
We may write
\begin{align*}
    \mathbf{E}_2(\mathbf{s})
    =\mathbf{E}_2(\mathbf{s})_{-}\cdot  u(-\mathbf{C}(\mathbf{s})^{t}\cdot (\mathbf{D}(\mathbf{s})^{t})^{-1}),
\end{align*}
where
\begin{align*}
    \mathbf{E}_2(\mathbf{s})_{-}=\begin{bmatrix}
        \mathbf{A}(\mathbf{s})& \mathbf{0}\\
        \mathbf{C}(\mathbf{s}) & \mathbf{D}(\mathbf{s})-\mathbf{C}\mathbf{A}(\mathbf{s})^{-1}\mathbf{B}(\mathbf{s})
    \end{bmatrix}.
\end{align*}
By Lemma \ref{two asymptotic parallel trajectories converge to the same}, for any $\mathbf{s}\in \mathcal{U}$,
$\mathbf{E}_1(\mathbf{s})(Id,\boldsymbol{\varphi}(\mathbf{s}))\Gamma$ is Birkhoff generic with respect to $(X,\mu_X,a_t)$ if and only if
\begin{align*}
    u(-\mathbf{C}(\mathbf{s})^{t}\cdot (\mathbf{D}(\mathbf{s})^{t})^{-1})\cdot\mathbf{E}_1(\mathbf{s}_0)(Id,\boldsymbol{\varphi}(\mathbf{s}))\Gamma
\end{align*}
is Birkhoff generic with respect to $(X,\mu_X,a_t)$.

By assumption, the map 
\begin{align*}
    \mathbf{s}\mapsto \mathbf{E}_2(\mathbf{s})^{-1}\cdot [\mathbf{e}_{r+1},\cdots,\mathbf{e}_d]=\begin{bmatrix}
        \mathbf{C}(\mathbf{s})^t\\
        \mathbf{D}(\mathbf{s})^t
    \end{bmatrix}
\end{align*}
has nonsingular differential at $\mathbf{s}_0$.
Thus the map 
\begin{align}\label{map from s to tilde s in compact group orbit setting}
   \phi: \mathbf{s}\mapsto -\mathbf{C}(\mathbf{s})^t\cdot (\mathbf{D}(\mathbf{s})^t)^{-1}
\end{align}
also has nonsingular differential at $\mathbf{s}_0$. 

Shrink the neighborhood $\mathcal{V}$ of $\mathbf{s}_0$ if necessary, we can assume that there exists an open subset $\Tilde{\mathcal{V}}$ of $Mat_{r\times(d-r)}(\bR)$ such that $\phi:\mathcal{V}\to \Tilde{\mathcal{V}}$ is a diffeomorphism. Denote $\phi^{-1}$ the inverse of $\phi$.

Let $\Tilde{\boldsymbol{\varphi}}(\Tilde{\mathbf{s}})=\boldsymbol{\varphi}(\phi^{-1}(\Tilde{\mathbf{s}}))$.
Applying Theorem \ref{theorem:genericity for general trajectory} to
\begin{align*}
    \{u(\Tilde{\mathbf{s}})\mathbf{E}_1(\mathbf{s}_0)(Id,\Tilde{\boldsymbol{\varphi}}(\Tilde{\mathbf{s}}))\Gamma:\Tilde{\mathbf{s}}\in \Tilde{\mathcal{V}}\},
\end{align*}
we obtain that if for any $\mathbf{m}\in \bZ^{k}\setminus \{\mathbf{0}\}$,
\begin{align}\label{condition for tilde s in compact group orbit setting}
    |\{\Tilde{\mathbf{s}}\in \Tilde{\mathcal{V}}: \Tilde{\boldsymbol{\varphi}}(\Tilde{\mathbf{s}})\cdot \mathbf{m} \in \mathbf{E}_1(\mathbf{s}_0)^{-1}\cdot u(-\Tilde{\mathbf{s}})\cdot \begin{bmatrix}
        \mathbf{0}\\
        \bR^{d-r}
    \end{bmatrix}
    +\mathbb{Z}^d\}|=0,
\end{align}
then for a.e. $\Tilde{\mathbf{s}}\in \Tilde{\mathcal{V}}$, $u(\Tilde{\mathbf{s}})\mathbf{E}_1(\mathbf{s}_0)(Id,\Tilde{\boldsymbol{\varphi}}(\Tilde{\mathbf{s}}))\Gamma$ is Birkhoff generic with respect to $(X,\mu_X,a_t)$. Since $\phi$ is a diffeomorphism, (\ref{condition for tilde s in compact group orbit setting}) is equivalent to 
\begin{align*}
    |\{\mathbf{s}\in \mathcal{V}:\boldsymbol{\varphi}(\mathbf{s})\cdot \mathbf{m}\in \mathbf{E}_1(\mathbf{s})^{-1}\cdot \begin{bmatrix}
        \mathbf{0}\\
        \bR^{d-r}
    \end{bmatrix}
    +\mathbb{Z}^d\}|=0.
\end{align*}
This completes the proof.
\end{proof}

\section{Application to universal hitting time statistics for integrable flows}\label{section: application to universal hitting time}

\subsection{An adapted form of Corollary \ref{genericity on orbit of maximal compact group}}

Following notations of \cite{Marklof_Universal_hitting_time_2017}, for $l>0$, let 
\begin{align*}
    D(e^{-l})=diag[e^{-(d-1)l},e^l,\cdots,e^l].
\end{align*}

\begin{theorem}\label{theorem adapted form of genericity}
Let $\mathcal{U}$ be a bounded open subset of $\bR^{d-1}$ and $\varphi:\mathcal{U}\to (\bR^d)^k$ be a $C^1$ map. Let $\mathbf{E_1}:\mathcal{U}\to SO_d(\bR)$ be a smooth map such that the map $\mathbf{s}\mapsto \mathbf{E_1}(\mathbf{s})^{-1}\cdot \mathbf{e}_1$ has a nonsingular differential at Lebesgue almost every $\mathbf{s}\in \mathcal{U}$. Assume that for any $\mathbf{m}\in \bZ^{k}\setminus\{\mathbf{0}\}$,
\begin{align*}
    |\{\mathbf{s}\in \mathcal{U}:\varphi(\mathbf{s})\cdot\mathbf{m}\in \bR\mathbf{E}_1(\mathbf{s})^{-1} \cdot \mathbf{e}_1
    +\mathbb{Z}^d \}|=0.
\end{align*}
Then for Lebesgue a.e. $\mathbf{s}\in \mathcal{U}$, $\mathbf{E}_1(\mathbf{s})(Id,\varphi(\mathbf{s}))\Gamma$ is Birkhoff generic with respect to $(X,\mu_X,D(e^{-l}))$.
\end{theorem}

\begin{proof}
For any $l>0$, denote  $a_l=diag[e^l,\cdots,e^l,e^{-(d-1)l}]$.
Choose $\omega\in SO_d(\bR)\cap SL_d(\bZ)$ such that for any
\begin{align*}
    \omega^{-1}\cdot D(e^{-l})\cdot\omega=a_l.
\end{align*}
Note that since $\omega\in SL_d(\bZ)$,
\begin{align*}
 \omega^{-1} D(e^{-l})E_1(\mathbf{s})(Id,\varphi_(\mathbf{s}))\Gamma =a_l \omega^{-1}E_1(\mathbf{s})\omega (Id,\omega^{-1}\varphi(\mathbf{s}))\Gamma.
\end{align*}

Applying Corollary \ref{genericity on orbit of maximal compact group}, we obtain that for a.e. $\mathbf{s}\in \mathcal{U}$,  $\omega^{-1}E_1(\mathbf{s})\omega(Id,\omega^{-1}\varphi(\mathbf{s}))\Gamma$ is Birkhoff generic with respect to $(X,\mu_X,a_l)$, and the theorem follows.
\end{proof}

\subsection{Universal hitting time}
{Let $(\mathcal{M},\mathcal{B},\nu)$ be a measurable space with probability measure $\nu$, and $\varphi^t:\mathcal{M}\to \mathcal{M}$ be a measure-preserving dynamical system. Given some target set $\mathcal{D}\subset \mathcal{M}$, it is natural to study how often a $\varphi$-trajectory along random initial data $x\in \mathcal{M}$ intersects this target set. On the other hand, another question is to consider a sequence of randomized target sets whose "size" shrink to zero, and study the distribution of intersection times of a random $\varphi$-trajectory with these shrinking targets. The interested reader is referred to \cite{Marklof_Universal_hitting_time_2017} and the references therein for a survey of the history of aforementioned questions.}

{
In the setting of universal hitting time statistics for integrable flows (cf. \cite{Marklof_Universal_hitting_time_2017}), the above question is studied when the measurable space is a $d$ dimensional torus $\mathbb{T}^d$ and $\varphi^t$ is a linear flow on the torus. 
} {Now let $d\geq 2$ and $k\geq 1$ be fixed integers. In this article, the sequence of target sets we consider is a sequence of union of $k$ many bounded codimensional one balls in $\mathbb{T}^d$, whose radius shrink to zero.} More precisely, let $\mathcal{U}$ be a bounded open subset of $\bR^{d-1}$. Consider the following smooth functions:
\begin{align*}
    \boldsymbol{\theta},\boldsymbol{\phi}_j: \mathcal{U}\to \mathbb{T}^d, 1\leq j\leq k,\\
    \mathbf{f},\mathbf{u}_j:\mathcal{U}\to \mathbf{S}_1^{d-1}, 1\leq j\leq k,
\end{align*}
where $\mathbf{S}_1^{d-1}$ is the unit one sphere in $\bR^d$. For the functions above, we assign to any $\mathbf{s}\in \mathcal{U}$ the following:
\begin{itemize}
 \item $\boldsymbol{\theta}(\mathbf{s})=$initial position of the flow;

\item $\mathbf{f}(\mathbf{s})=$direction of the flow;

\item $\mathbf{u}_j(\mathbf{s})=$direction of the $j$-th target ball;

\item $\boldsymbol{\phi}_j(\mathbf{s})=$center of the $j$-th target ball.
\end{itemize}
With these functions, we define the flow 
\begin{align}\label{definition flow map in universal hitting time}
    \varphi^t:\mathcal{U}\to \mathbb{T}^d\times \mathcal{U}, \mathbf{s}\mapsto (\boldsymbol{\theta}(\mathbf{s})+t\mathbf{f}(\mathbf{s}),\mathbf{s}).
\end{align}
From now on, we fix a map
$\mathbf{v}\mapsto \mathbf{R}_{\mathbf{v}}$ from $\mathbf{S}_1^{d-1}$ to $SO_d(\bR)$ such that for all $\mathbf{v}\in \mathbf{S}_1^{d-1}$,
\begin{align}\label{property of the map R}
    \mathbf{R}_{\mathbf{v}}\cdot\mathbf{v}=\mathbf{e}_1,
\end{align}
and $\mathbf{v}\mapsto \mathbf{R}_{\mathbf{v}}$ is smooth on $\mathbf{S}_1^{d-1}\setminus\{\mathbf{v}_0\}$ for a singular point $\mathbf{v}_0\in \mathbf{S}_1^{d-1}$. {
For $1\leq j \leq k$, fix a bounded open subset $\Omega_j\subset \bR^{d-1}\times \mathcal{U}$. For any $l>0$, denote the $l$-level target set with to be}
\begin{align*}
    \mathcal{D}_l=\bigcup_{j=1}^k \mathcal{D}_l(\mathbf{u}_j,\boldsymbol{\phi}_j,\Omega_j),
\end{align*}
where
\begin{align*}
    \mathcal{D}_l(\mathbf{u}_j,\boldsymbol{\phi}_j,\Omega_j)=\{(\boldsymbol{\phi}_j(\mathbf{s})+e^{-l} \mathbf{R}^{-1}_{\mathbf{u}_j(\mathbf{s})}\cdot \begin{bmatrix}
        0\\
        \mathbf{x}
    \end{bmatrix},\mathbf{s})\in \mathbb{T}^d\times \mathcal{U}:(\mathbf{x},\mathbf{s})\in \Omega_j\}.
\end{align*}
Note that $l>0$ parameterizes the size of the target. For any $\mathbf{s}\in \mathcal{U}$, let $\mathcal{T}(\mathbf{s},\mathcal{D}_l)$ be the set of hitting times defined by 
\begin{align*}
    \mathcal{T}(\mathbf{s},\mathcal{D}_l):=\{t>0:\varphi^t(\mathbf{s})\in \mathcal{D}_l\}.
\end{align*}
This is a discrete subset of $\bR_{>0}$, and we label its elements by 
\begin{align*}
    0<t_1(\mathbf{s},\mathcal{D}_l)<t_2(\mathbf{s},\mathcal{D}_l)<\cdots.
\end{align*}
By Santalo's formula (cf. \cite{Chernov_Entropy_Lyapunov_exponents_1997}), if $\mathbf{s}\in \mathcal{U}$ is such that the components of $\mathbf{f}(\mathbf{s})$ are not rationally related, then for any $n\in \mathbb{N}$, the normalized $n$-th return time to target $\mathcal{D}_l$ is
\begin{align*}
    \frac{t_n(\mathbf{s},\mathcal{D}_l)}{e^{(d-1)l}\cdot\overline{\sigma}(\mathbf{s})},
\end{align*}
where {$e^{(d-1)l}\cdot\overline{\sigma}(\mathbf{s})$ is the mean return time (cf. \cite[Section 2]{Marklof_Universal_hitting_time_2017}), and }
\begin{align*}
    \overline{\sigma}(\mathbf{s})=\frac{1}{\sum_{j=1}^k |\Omega_j(\mathbf{s})|\mathbf{u}_j(\mathbf{s})\cdot \mathbf{f}(\mathbf{s})},\textbf{  } \Omega_j(\mathbf{s})=\{\mathbf{x}\in \bR^{d-1}:(\mathbf{x},\mathbf{s})\in \Omega_j\}.
\end{align*}

\begin{definition}\label{definition regular f}
The smooth map $\mathbf{f}:\mathcal{U}\to \mathbf{S}_1^{d-1}$ is regular if the push forward of Lebesgue measure on $\mathcal{U}$ under $\mathbf{f}$
is absolutely continuous with respect to the Haar measure on $\mathbf{S}_1^{d-1}$.
\end{definition}

The following is a corollary of Theorem \ref{theorem adapted form of genericity}:
\begin{corollary}\label{corollary adapted form for genericity of universal hitting time}
Let $\mathcal{U}$ be a bounded open subset of $\bR^{d-1}$. Let $\mathbf{f}:\mathcal{U}\to \mathbf{S}_1^{d-1}$ be a regular smooth map. Let $\boldsymbol{\varphi}:\mathcal{U}\to (\bR^d)^k$ be a $C^1$ map. If for any $\mathbf{m}\in \bZ^k\setminus\{\mathbf{0}\}$,
\begin{align}\label{genericity condition when flow presented}
    |\{\mathbf{s}\in\mathcal{U}:\boldsymbol{\varphi}(\mathbf{s})\cdot \mathbf{m}\in \bR \mathbf{f}(\mathbf{s})+\mathbb{Z}^d\}|=0,
\end{align}
then for Lebesgue a.e. $\mathbf{s}\in \mathcal{U}$, $\mathbf{R}_{\mathbf{f}(\mathbf{s})}(Id,\boldsymbol{\varphi}(\mathbf{s}))\Gamma$ is Birkhoff generic with respect to $(X,\mu_X,D(e^{-l}))$.
\end{corollary}

\begin{proof}
By assumption, $\mathbf{f}$ is a regular smooth map from $\mathcal{U}$ to $\mathbf{S}_1^{d-1}$. By Sard's theorem, the set of critical points of $\mathbf{f}$ has Lebesgue measure $0$.
For any $\mathbf{s}_0\in \mathcal{U}$ which is not a critical point of $\mathbf{f}$, there exists an open neighborhood $\mathcal{V}$ of $\mathbf{s}_0$ such that $\mathbf{f}$ is a diffeomorphism of $\mathcal{V}$ to some open subset of $\mathbf{S}_1^{d-1}$.

Therefore, the map $\mathbf{s}\mapsto \mathbf{R}_{\mathbf{f}(\mathbf{s})}^{-1}\cdot \mathbf{e}_1=\mathbf{f}(\mathbf{s})$ has nonsingular differentials for all $\mathbf{s}\in \mathcal{V}$. Then we apply Theorem \ref{theorem adapted form of genericity} to $\mathbf{E}_1(\mathbf{s})=\mathbf{R}_{\mathbf{f}(\mathbf{s})}$ and the corollary follows.
\end{proof}

\begin{definition}\label{theta generic}
The $k$-tuple of smooth functions $\boldsymbol{\phi}_1,\cdots \boldsymbol{\phi}_k:\mathcal{U}\to \mathbb{T}^d$ is $\boldsymbol{\theta}$-generic if for any $\mathbf{m}=(m_1,\cdots,m_k)\in \bZ^k\setminus\{\mathbf{0}\}$, we have 
\begin{align*}
    |\{\mathbf{s}\in \mathcal{U}: \sum_{j=1}^k m_j(\boldsymbol{\phi}_j(\mathbf{s})-\boldsymbol{\theta}(\mathbf{s}))\in \bR \mathbf{f}(\mathbf{s})+\mathbb{Z}^d\}|=0.
\end{align*}
\end{definition}

To state our theorem precisely, we need some preparations. Our notations follows from \cite[Section 6]{Marklof_Universal_hitting_time_2017}. Given $N\in \mathbb{N}$, denote $\overline{N}=\{1,\cdots,N\}$. For $j\in \{1,\cdots,k\}$ and $\mathbf{s}\in \mathcal{U}$, define $\mathfrak{R}_j(\mathbf{s})=\mathbf{R}_{\mathbf{f}(\mathbf{s})}\mathbf{R}^{-1}_{\mathbf{u}_j(\mathbf{s})}$.
Let $\Tilde{\mathfrak{R}}_{j}(\mathbf{s})$ be the matrix of the linear transformation
\begin{align}\label{definition of tilde R}
    \mathbf{x}\mapsto (\mathfrak{R}_j(\mathbf{s})\begin{bmatrix}
    0\\
    \mathbf{x}
\end{bmatrix})_{\perp}\in\bR^{d-1},
\end{align}
where $\mathbf{u}_{\perp}=(u_2,\cdots,u_d)^{tr}\in \bR^{d-1}$ for $\mathbf{u}=(u_1,\cdots,u_d)^{tr}\in \bR^d$.

For any $\mathbf{s}\in \mathcal{U}$, we define
\begin{align}\label{definition of tilde Omega}
    \Tilde{\Omega}_j(\mathbf{s}):=\overline{\sigma}(\mathbf{s})^{\frac{1}{d-1}}\Tilde{\mathfrak{R}}_j(\mathbf{s})\Omega_j(\mathbf{s})\subset \bR^{d-1}.
\end{align}
Let $G_1=SL_d(\bR)\ltimes \bR^d$. For $g=(g\prm,(\boldsymbol{\xi}_1,\cdots,\boldsymbol{\xi}_k))\in G$ and $j\in \{1,\cdots,k\}$, write $g^{[j]}=(g\prm,\boldsymbol{\xi}_j)\in G_1$. {Our main theorem of this section is the following.}
\begin{theorem}\label{theorem on universal hitting time}
Let $\mathcal{U}$ be a bounded open subset of $\bR^{d-1}$. For $1\leq j\leq k$, let $\mathbf{f},\boldsymbol{\theta},\mathbf{u}_j,\boldsymbol{\phi}_j$ be given as in the beginning of this section. Let $\Omega_j$ be a bounded open subset of $\bR^{d-1}\times \mathcal{U}$. For each $j=1,\cdots, k$, assume that

(1) $|\mathbf{u}_j^{-1}(\{\mathbf{v}_0\})|=0$,

(2) $\mathbf{u}_j(\mathbf{s})\cdot \mathbf{f}(\mathbf{s})>0$ for all $\mathbf{s}\in \mathcal{U}$,

(3) for a.e. $\mathbf{s}\in \mathcal{U}$, the boundary $\partial \Omega_j(\mathbf{s})$ has Lebesgue measure $0$.

(4) $|\Omega_j(\mathbf{s})|$ is a smooth positive function of $\mathbf{s}\in \mathcal{U}$.

Also assume that $\mathbf{f}$ is regular and $(\boldsymbol{\phi}_1,\cdots,\boldsymbol{\phi}_k)$ is $\boldsymbol{\theta}$-generic. Then for any $N\in \mathbb{N}$, any $T_n>0$ for $n\in \overline{N}$, the following holds: For a.e. $\mathbf{s}\in \mathcal{U}$, 
\begin{align*}
    &\lim_{L\to \infty}\frac{1}{L} |\{l\in [0,L]: \frac{t_n(\mathbf{s},\mathcal{D}_l)}{e^{(d-1)l}\cdot\overline{\sigma}(\mathbf{s})}\leq T_n,\forall n\in \overline{N}\}|\\
   &= \mu_X(\{g\Gamma\in X:
    \sum_{j=1}^k\#\left\{\begin{bmatrix}
        t\\
        \mathbf{x}
    \end{bmatrix}\in g^{[j]}\bZ^d:0<t<T_n,\mathbf{x}\in -\Tilde{\Omega}_j(\mathbf{s})\right\}\geq n,\\
    &\forall n\in \overline{N}\}).
\end{align*}
\end{theorem}
{Note that in \cite[Theorem 2]{Marklof_Universal_hitting_time_2017}, the authors proved that for each $n\in \mathbb{N}$, there is some random variable $\tau_n$ in $\mathbb{R}_{>0}$ such that the $n$-th normalized hitting time $t_n(\cdot,\mathcal{D}_l)/(e^{(d-1)l}\cdot\overline{\sigma}(\cdot))$ converges to $\tau_n$ in distribution as $l\to \infty$. Unlike \cite{Marklof_Universal_hitting_time_2017}, in Theorem \ref{theorem on universal hitting time} we are interested in the question that given a \textbf{fixed} initial $\mathbf{s}$, when the target is shrinking, how often the $n$-th normalized hitting time is bounded by some given constant.}

For any $j\in \{1,\cdots,k\}$, any real numbers $Y<Z$, following \cite[Eq. (8.9)]{Marklof_Universal_hitting_time_2017}, we define 
\begin{align*}
    \Tilde{\mathbf{A}}_{j,Y,Z}=\{(\begin{bmatrix}
        t\\
        -\Tilde{\mathfrak{R}}_j(\mathbf{s})\mathbf{x}
    \end{bmatrix}
    ,\mathbf{s}): (\mathbf{x},\mathbf{s})\in \Omega_j, \overline{\sigma}(\mathbf{s})Y< t \leq \overline{\sigma}(\mathbf{s})Z\}.
\end{align*}
Given any real numbers $Y_n<Z_n$ for $n\in \overline{N}$, following \cite[Eq.(8.10)]{Marklof_Universal_hitting_time_2017}, we define 
\begin{align*}
    &B[(Y_n),(Z_n)]\nonumber
    =\{(g\Gamma,\mathbf{s})\in G/\Gamma\times \mathcal{U}:\sum_{j=1}^k \#(\Tilde{A}_{j,Y_n,Z_n}(\mathbf{s})\cap g^{[j]}\cdot \bZ^d)\geq n, \forall n\in \overline{N}\}.
\end{align*}
where 
\begin{align*}
   \Tilde{A}_{j,Y_n,Z_n}(\mathbf{s})= \{\mathbf{x}\in \bR^d: (\mathbf{x},\mathbf{s})\in \Tilde{A}_{j,Y_n,Z_n}\}.
\end{align*}
For any $\mathbf{s}\in \mathcal{U}$, denote
\begin{align*}
    B[(Y_n),(Z_n)](\mathbf{s}):=\{g\Gamma\in G/\Gamma:(g\Gamma,\mathbf{s})\in B[(Y_n),(Z_n)]\}.
\end{align*}

\begin{lemma}\cite[Lemma 17]{Marklof_Universal_hitting_time_2017}\label{lemma intersection nonempty}\label{lemma boundary has measure 0}
For every $\mathbf{s}\in \mathcal{U}$, and $B=B[(Y_n),(Z_n)]$, $\mu_X(\partial B(\mathbf{s}))=0$.
\end{lemma}
\begin{proof}
By \cite[Lemma 14,16]{Marklof_Universal_hitting_time_2017}, it suffices to prove that for every $j\in \{1,\cdots,k\}$ and $n\in \{1,\cdots, N\}$, $\partial \Tilde{A}_{j,Y_n,Z_n}(\mathbf{s})$ has Lebesgue measure zero. Now since for any $Y<Z$, we have 
\begin{align*}
    \partial \Tilde{A}_{j,Y,Z}(\mathbf{s})&=\{\begin{bmatrix}
        t\\
        -\Tilde{\mathfrak{R}}_j(\mathbf{s})x
    \end{bmatrix}: \mathbf{x}\in \partial \Omega_j(\mathbf{s}), \overline{\sigma}(\mathbf{s})Y<t\leq \overline{\sigma}(\mathbf{s})Z\}\\
    &\bigcup \{\begin{bmatrix}
        t\\
        -\Tilde{\mathfrak{R}}_j(\mathbf{s})x
    \end{bmatrix}: \mathbf{x}\in \overline{\Omega_j(\mathbf{s})},t\in \{\overline{\sigma}(\mathbf{s})Y,\overline{\sigma}(\mathbf{s})Z\}\}.
\end{align*}
By assumption, for a.e. $\mathbf{s}\in \mathcal{U}$, $|\partial \Omega_j(\mathbf{s})|=0$, the lemma follows.
\end{proof}

We are now ready to prove Theorem \ref{theorem on universal hitting time}.

\begin{proof}[Proof of Theorem \ref{theorem on universal hitting time}]
{The proof of Theorem \ref{theorem on universal hitting time} is almost the same as the proof of \cite[Theorem 2]{Marklof_Universal_hitting_time_2017}, except that here we use the equidistribution result for the average along $a_t$ trajectory, while in \cite{Marklof_Universal_hitting_time_2017}, the equidistribution of $a_t$ translation of the average over a bounded open subset in horospherical subgroup is used.}

Let $\Tilde{\boldsymbol{\varphi}}:\mathcal{U}\to (\bR^d)^k$ be a map given by 
\begin{align*}
    \Tilde{\boldsymbol{\varphi}}(\mathbf{s})=(\boldsymbol{\phi}_1(\mathbf{s})-\boldsymbol{\theta}(\mathbf{s}),\cdots,\boldsymbol{\phi}_k(\mathbf{s})-\boldsymbol{\theta}(\mathbf{s})).
\end{align*}
Since $(\boldsymbol{\phi}_1,\cdots,\boldsymbol{\phi}_k)$ is $\boldsymbol{\theta}$-generic, $\mathbf{f}$ is regular, assumption (\ref{genericity condition when flow presented}) in Corollary \ref{corollary adapted form for genericity of universal hitting time} are satisfied for the maps $\Tilde{\boldsymbol{\varphi}}$ and $\mathbf{f}$, thus Corollary \ref{corollary adapted form for genericity of universal hitting time} applies.

Let $B=B[(Y_n),(Z_n)]$ for $Y_n,Z_n\in \bR$ and $n\in \overline{N}$.
Since by Lemma \ref{lemma boundary has measure 0}, $\mu_X(\partial B(\mathbf{s}))=0$, for all $\mathbf{s}\in \mathcal{U}$, we have for a.e. $\mathbf{s}\in \mathcal{U}$,
\begin{align*}
    \lim_{L\to \infty}\frac{1}{L}\int_0^{L}\chi_{B(\mathbf{s})}(D(e^{-l})\mathbf{R}_{\mathbf{f}(\mathbf{s})}(Id,\Tilde{\varphi}(\mathbf{s})))dl=\mu_X(B(\mathbf{s})).
\end{align*}
Then the rest of the proof follows from the proof of \cite[Theorem 2]{Marklof_Universal_hitting_time_2017}.
\end{proof}

\noindent$\textbf{Acknowledgment.}$ I would like to thank Yitwah Cheung,  Nimish Shah, Ronggang Shi and Barak Weiss for helpful discussions and comments. I thank Jim Cogdell for pointing out some inaccuracies in the draft. I also thank anonymous referee for pointing out many inaccuracies of the paper, as well as providing a lot of helpful suggestions to improve this article.

\end{document}